\documentclass[a4paper, 12pt]{article} 
\usepackage{amssymb,latexsym} \author{Johannes Sj\"ostrand\\
\small IMB, Universit\'e de Bourgogne\\
\small 9, Av.~A.~Savary, BP 47870\\
\small FR-21078 Dijon c\'edex\\
\footnotesize johannes.sjostrand@u-bourgogne.fr\\
\footnotesize  and UMR 5584, CNRS} \date{}
\title{Eigenvalue distribution for non-self-adjoint operators with
  small multiplicative random perturbations}

\newtheorem{dref}{Definition}[section] \newtheorem{lemma}[dref]{Lemma}
\newtheorem{theo}[dref]{Theorem} \newtheorem{prop}[dref]{Proposition}
\newtheorem{remark}[dref]{Remark} \newtheorem{ex}[dref]{Example}
\newtheorem{cor}[dref]{Corollary}

\newenvironment{proof}{\par\noindent{{\bf Proof}}}{\hfill$\Box$
\medskip}

\newcommand{\ekv}[2]{\begin{equation}\label{#1}#2\end{equation}}
\newcommand{\eekv}[3]{\begin{eqnarray}\label{#1}#2 \\ #3
\nonumber\end{eqnarray}}
\newcommand{\eeekv}[4]{\begin{eqnarray}\label{#1}#2 \\ #3
\nonumber\\#4\nonumber\end{eqnarray}}

  \newcommand\iint{\int\hskip -2mm\int}

\newcommand{\no}[1]{(\ref{#1})} \newcommand\trans[1]{{^t\hskip -2pt
#1}}
\begin{document}

\maketitle
\begin{abstract} In this work we continue the study of the
Weyl asymptotics of the distribution of eigenvalues of
non-self-adjoint (pseudo)diff\-erential operators with small random
perturbations, by treating the case of multiplicative perturbations in
arbitrary dimension. We were led to quite essential improvements of
many of the probabilistic aspects.

\medskip
\par \centerline{\bf R\'esum\'e} Dans ce travail nous continuons 
l'\'etude de l'asymp\-totique de Weyl de la
distribution des valeurs propres d'op\'erateurs
(pseudo-)diff\'erentiels avec des perturbations al\'eatoires petites,
en traitant le cas des perturbations multiplicatives en dimension 
quelconque. Nous avons \'et\'e amen\'es \`a faire des 
am\'eliorations essentielles des aspects probabilistes.

\end{abstract}

\tableofcontents

\section{Introduction}\label{int}
\setcounter{equation}{0}
In \cite{Ha} Mildred Hager considered a class of randomly perturbed
semi-classical unbounded (pseudo-)differential operators of the form 
\ekv{int.1}
{
P_\delta =P(x,hD_x;h)+\delta Q_\omega ,\ 0<h\ll 1,
}
on $L^2({\bf R})$, where $P(x,hD_x;h)$ is a non-self-adjoint
pseudodifferential operator of some suitable class (including
differential operators) with leading symbol $p(x,\xi )$ and where
$Q_\omega u(x)=q_\omega (x)u(x)$ is a random multiplicative
perturbation and  $\delta >0$ is a
small parameter. 

\par Let $\Gamma \Subset {\bf C}$ have smooth boundary and
assume that $p^{-1}(z)$ is finite for every $z\in \Gamma $ and also
that $\{ p,\overline{p}\}(\rho ) \ne 0$ for every $\rho \in
p^{-1}(\Gamma )$. Then under some additional assumptions
Hager showed that for $\delta =e^{-\epsilon /h}$,
the number $\# (\sigma (P_\delta )\cap \Gamma )$ of eigenvalues of $P_\delta $ in
$\Gamma $ satisfies
\ekv{int.3}
{
|\# (\sigma (P_\delta )\cap \Gamma )-\frac{1}{2\pi h}\mathrm{vol\,}(p^{-1}(\Gamma ))|
\le \frac{C\sqrt{\epsilon }}{h},
}
with a probability very close to 1 in the limit of small $h$. 

\par Recently, W. Bordeaux-Montrieux \cite{Bo} established almost sure
Weyl asymptotics for the large eigenvalues of elliptic operators and
systems on $S^1$ under assumptions quite similar to those of
Hager. The one-dimensional nature of the problems is essential in the 
proofs in \cite{Ha, Bo}. 

\par In \cite{HaSj}, Hager and the author found a new approach and
extended the results of \cite{Ha} to the case of operators on ${\bf
R}^n$ and replaced the assumption about the non-vanishing of $\{
p,\overline{p}\} $ by a  weaker condition, allowing $\Gamma $ to
contain boundary points of $\overline{p({\bf R}^{2n}})$. In dimension
$\ge 2$, it turned out to be  simpler to consider general random
perturbations of the form
\ekv{int.4}{\delta Q_\omega u=\delta \sum\sum \alpha
  _{j,k}(\omega )(u| f_k) e_j ,}
where $\{ e_j\}$, $\{ f_k\} $ are orthonormal families of
eigenfunctions of certain elliptic $h$-pseudodifferential operators of
Hilbert Schmidt class and $\alpha _{j,k}(\omega )$ are independent
complex Gaussian random variables. With some
exageration, the results of \cite{HaSj} show that most non-self-adjoint
pseudodifferential operators obey Weyl-asymptotics, but since the 
perturbations are no
more multiplicative, we did not have the same conclusion for the 
differential operators.

\par The purpose of the present paper is to treat the case of
multiplicative perturbations in any dimension. 
Several elements of \cite{HaSj} carry over
to the multiplicative case, while the study of a certain effective
Hamiltonian, here a finite random matrix, turned out to be more
difficult. Because of that we were led to abandon the fairly explicit
calculations with Gaussian random variables and instead resort to
arguments from complex analysis. A basic difficulty was then to find at
least one perturbation within the class of permissible ones, for which
we have a lower bound on the determinant of the associated effective
Hamiltonian. This is achieved via an iterative (``renormalization'')
procedure, with estimates on the singular values at each step. An
advantage with the new approach is that we can treat
 more general random perturbations.

We next state the main result of this work. For
simplicity we shall work on ${\bf R}^n$, where some results from
\cite{HaSj} are already available. In principle the extension of our
results to the case of compact manifolds should only present moderate 
technical difficulties. 

\par Let us first specify the assumptions about the unperturbed
operator. 

\par Let $m\ge 1$ be an order function on ${\bf R}^{2n}$
in the sense that 
$$
m(\rho )\le C_0 \langle \rho -\mu \rangle^{N_0}m(\mu ),\ \rho , \mu\in 
{\bf R}^{2n}  
$$
for some fixed positive constants $C_0,N_0$, where we use the standard
notation $\langle \rho \rangle= (1+|\rho |^2)^{1/2}$.

Let 
$$p\in
S(m):=\{ a\in C^\infty ({\bf R}^{2n});\, \vert \partial _\rho ^\alpha
a(\rho )\vert\le C_\alpha m(\rho ),\, \forall \rho \in {\bf
  R}^{2n},\alpha \in {\bf N}^{2n}\} .$$ 

We assume that $p-z$ is elliptic (in the sense that $(p-z)^{-1}\in S(m^{-1})$) for at least one value
$z\in {\bf C}$. Put $\Sigma =\overline{p({\bf R}^{2n})}=p({\bf
  R}^{2n})\cup \Sigma _\infty $, where $\Sigma _\infty $ is the set
of accumulation values of $p(\rho )$ near $\rho = \infty $.  Let 
$P(\rho )=P(\rho ;h)$, $0<h\le h_0$ belong to $S(m)$ in the sense that
$|\partial _\rho ^\alpha P(\rho ;h)| \le C_\alpha m(\rho ) $ as above,
with contants that are independent of $h$. Assume that there exist
$p_1,p_2,...\in S(m)$ such that
$$P\sim p+hp_1+...
\hbox{ in } S(m),\ h\to 0.$$
By $P=P(x,hD_x;h)$ we also denote the Weyl
quantization of $P(x,h\xi ;h)$ (see for instance \cite{DiSj}).
 Let $\Omega \Subset {\bf C}$ be open simply connected
with $\overline{\Omega }\cap \Sigma _\infty =\emptyset$, $\Omega
\not\subset \Sigma $.  
Then for $h>0$ small enough, the spectrum $\sigma (P)$ of $P$ is
discrete in $\Omega $ and constituted of eigenvalues of finite
algebraic multiplicity.
 We will also need the symmetry assumption,
\ekv{int.6}
{
P(x,-\xi ;h)=P(x,\xi ;h).
}

\par Let $V_z(t):=\mathrm{vol\,}(\{ \rho \in {\bf R}^{2n};\,|p(\rho
)-z|^2 \le t\} )$. For $\kappa \in ]0,1]$, $z\in \Omega $, we consider
the property that
\ekv{int.6.2}{V_z(t)={\cal O}(t^ \kappa ),\ 0\le t \ll 1.} 

\par Let $K$ be a compact neighborhood of $\pi _xp^{-1}(\Omega )$,
where $\pi _x$ denotes the na\-tural projection from the cotangent
bundle to the base space. 
The random potential will be of the form 
\ekv{int.6.3}
{q_\omega (x)=\chi _0(x)\sum_{0<\mu _k\le L}\alpha _k(\omega )\epsilon
_k(x),\ |\alpha |_{{\bf C}^D}\le R,}
where $\epsilon _k$ is the orthonormal basis of eigenfunctions of
$h^2\widetilde{R}$, where $\widetilde{R}$ is an $h$-independent
 positive elliptic 2nd order operator with smooth coefficients 
on a compact manifold of
dimension $n$, containing an open set diffeomorphic to an open
neighborhood of $\mathrm{supp\,}\chi _0$. Here $\chi _0\in C_0^\infty
({\bf R}^n)$ is equal to 1 near $K$. $\mu _k^2$ denote the
corresponding eigenvalues, so that $h^2\widetilde{R}\epsilon _k=\mu
_k^2 \epsilon _k$. We choose $L=L(h)$ and $R=R(h)$ in the intervals
\eekv{int.6.4}
{h^{\frac{\kappa -3n}{s-\frac{n}{2}-\epsilon }}\ll L\le C h^{-M},&&
M\ge \frac{3n-\kappa }{s-\frac{n}{2}-\epsilon },}
{\frac{1}{C}h^{-(\frac{n}{2}+\epsilon )M+\kappa -\frac{3n}{2}}\le R\le
 C h^{-\widetilde{M}},&& \widetilde{M}\ge \frac{3n}{2}-\kappa
  +(\frac{n}{2}+\epsilon )M,}
for some $\epsilon \in ]0,s-\frac{n}{2}[$, $s>\frac{n}{2}$,
so by Weyl's law for the large eigenvalues of elliptic
self-adjoint operators, the dimension $D$ is of the order of magnitude
${\cal O}((L/h)^n)$. We introduce the  small parameter\footnote{In the proof of the main result, we get $\delta
  =\tau_0h^{N_1+n}/C$ for some large constant $C$, but a dilation in
  $\tau_0$ can easily be absorbed in the constants later on.} 
$\delta =\tau _0 h^{N_1+n}$, $\tau_0=\tau_0(h)\in ]0,\sqrt{h}]$, 
where 
\ekv{int.6.4.3}
{N_1:=\widetilde{M}+sM+\frac{n}{2}.} 
The randomly perturbed operator is
\ekv{int.6.4.5}
{
P_\delta =P+\delta h^{N_1}q_\omega =:P+\delta Q_\omega .
}
We have chosen the exponent $N_1$ so that $\Vert
h^{N_1}q\Vert_{L^\infty }\le {\cal O}(1) h^{-n/2}\Vert h^{N_1}q
\Vert_{H^s}\le {\cal O}(1)$, when $q$ is an admissible potential as
in (\ref{int.6.3}), (\ref{int.6.4}) and $H^s$ is the semiclassical
Sobolev space in Section \ref{al}. The lower bounds on $L,R$ are
dictated by the construction of a special admissible potential in
Sections \ref{cl}, \ref{sv}.

\par The random variables $\alpha _j(\omega )$ will have a
joint probability distribution \ekv{int.6.5}{P(d\alpha )=C(h)e^{\Phi
(\alpha ;h)}L(d\alpha ),} where for some $N_4>0$,
\ekv{int.6.6}{ |\nabla _\alpha \Phi |_{{\bf C}^D}={\cal
O}(h^{-N_4}),} and $L(d\alpha )$ is the
Lebesgue measure. ($C(h)$ is the normalizing constant, 
assuring that the probability of
$B_{{\bf C}^D}(0,R)$ is equal to 1.) 

\par We also need the parameter 
\ekv{int.6.7.5}{\epsilon _0(h)=(h^{\kappa }+h^n\ln 
\frac{1}{h})(\ln \frac{1}{\tau _0}+(\ln \frac{1}{h})^2)} and assume
that $\tau _0=\tau _0(h)$ is not too small, so that $\epsilon _0(h)$ is
small. The main result of this work is:
\begin{theo}\label{int1} Under the assumptions above, let 
$\Gamma \Subset \Omega $ have smooth boundary, let $\kappa \in
]0,1]$ be the parameter in \no{int.6.3}, \no{int.6.4}, \no{int.6.7.5} and assume that 
\no{int.6.2} holds uniformly for $z$ in a
neighborhood of $\partial \Gamma $. Then there
exists a constant $C>0$ such that 
for $C^{-1}\ge r>0$,
$\widetilde{\epsilon }\ge C \epsilon _0(h)$ 
we have with probability 
\ekv{int.6.8}{
\ge 1-\frac{C\epsilon _0(h)}
{rh^{n+\max (n(M+1),N_4+\widetilde{M})}}
e^{-\frac{\widetilde{\epsilon }}{C\epsilon _0(h)}} }
that:
\eekv{int.7}
{
&&|
\#(\sigma (P_\delta )\cap \Gamma )-\frac{1}{(2\pi h)^n
}\mathrm{vol\,}(p^{-1}(\Gamma ))
|\le
}
{&&
\frac{C}{h^n}\left( \frac{\widetilde{\epsilon }}{r}
+C\left(r+\ln (\frac{1}{r})\mathrm{vol\,}(p^{-1}(\partial
\Gamma +D(0,r)))\right)
 \right).}
Here $\#(\sigma (P_\delta )\cap \Gamma )$ denotes the number of
eigenvalues of $P_\delta $ in $\Gamma $, counted with their algebraic multiplicity.

\end{theo}

Actually, we shall prove the theorem for the slightly more general
operators, obtained by replacing $P$ by $P_0$ in (\ref{sv.5.5}).

\par
The second volume in (\ref{int.7}) is ${\cal O}(r^{2\kappa -1})$ which
is of interest when $\kappa >1/2$. In that case
\ekv{int.7.1}
{
\ln \frac{1}{r}\mathrm{vol\,}(p^{-1}(\partial \Gamma +D(0,r))={\cal
  O}(r^\beta ),
}
for any $\beta \in ]0,2\kappa -1[$. Even if $\kappa <1/2$ we can
reasonably assume that (\ref{int.7.1}) holds for some $\beta >0$. (For
instance if $p$ is real-valued and $\Gamma $ does not contain any
critical values of $p$, then (\ref{int.6.2}) holds uniformly for $z$
in a neighborhood of $\partial \Gamma $ with $\kappa =1/2$, but if we
choose $\Gamma $ so that its boundary can only intersect the real axis
transversally, then $\mathrm{vol\,}(p^{-1}(\partial \Gamma
+D(0,r)))={\cal O}(r)$.) Assuming (\ref{int.7.1}) for some $\beta >0$
we choose $r=\widetilde{\epsilon }^{\frac{1}{\beta +1}}$ and the right
hand side of (\ref{int.7}) is $\le Ch^{-n}\widetilde{\epsilon }^{\beta
/(1+\beta )}$, which gives Weyl asymptotics, if $\widetilde{\epsilon
}$ is small.

\par If we assume that 
\ekv{int.7.2}
{
\exp (-h^{-\kappa _0})\le \tau_0\le \sqrt{h},\hbox{ for some }\kappa
_0\in ]0,\kappa [,
}
then 
\ekv{int.7.3}
{
\epsilon _0={\cal O}(h^{\kappa -\kappa _0}\ln \frac{1}{h})
}
is small. Now take $\widetilde{\epsilon }=h^{\widetilde{\kappa }}$,
for some $\widetilde{\kappa }\in ]0,\kappa -\kappa _0[$. Then, we get
the following corollary:
\begin{cor}\label{int1.5}
We make the general assumptions of Theorem \ref{int1}. Assume
(\ref{int.7.1}) for some $\beta >0$ and recall that this is automatically
the case when $\kappa >1/2$ and $0<\beta <2\kappa -1$. Choose
$\delta $ as prior to (\ref{int.6.4.5}) with $\tau_0$ as in 
(\ref{int.7.2}). Let $0<\widetilde{\kappa} <\kappa -\kappa _0$. Then,
with probability
\ekv{int.7.4}
{
\ge 1-\frac{Ch^{\kappa -\kappa _0}\ln
  \frac{1}{h}}{h^{\frac{\widetilde{\kappa }}{1+\beta
    }+n+\max (n(M+1),N_4+\widetilde{M})}}e^{-h^{\widetilde{\kappa
    }-(\kappa -\kappa _0)}/(C\ln \frac{1}{h})},
}
we have 
\ekv{int.7.5}
{
|\# (\sigma (P_\delta )\cap \Gamma )-\frac{1}{(2\pi
  h)^n}\mathrm{vol\,}(p^{-1}(\Gamma ))|\le
\frac{C}{h^n}h^{\frac{\widetilde{\kappa } \beta }{1+\beta }}.
}
\end{cor}

As in \cite{HaSj} we also have a result valid simultaneously for a
family ${\cal C}$ of domains $\Gamma \subset \Omega $ satisfying the
assumptions of Theorem \ref{int1} uniformly in the natural sense:
With a probability 
\ekv{int.8}{
\ge 1-\frac{{\cal O}(1)\epsilon _0(h)}{r^2h^{n+\max (n(M+1),N_4+\widetilde{M})}}e^{-\frac{\widetilde{\epsilon }}{C\epsilon _0(h)}}, } the
estimate \no{int.7} holds simultaneously for all $\Gamma \in {\cal
  C}$. The corresponding variant of Corollary \ref{int1.5} holds also;
just replace $\frac{\widetilde{\kappa }}{1+\beta }$ in the exponent of
the denominator in (\ref{int.7.4}) by 
$\frac{2\widetilde{\kappa }}{1+\beta }$.

\begin{remark}\label{int2}
{\rm When $\widetilde{R}$ has real coefficients, we may assume that the
eigenfunctions $\epsilon _j$ are real. Then (cf Remark \ref{pr3}) we may
restrict $\alpha $ in (\ref{int.6.3}) to be in ${\bf R}^D$ so that
$q_\omega $ is real, still with $|\alpha |\le R$, and change
$C(h)$ in (\ref{int.6.5}) so that $P$ becomes a probability measure on 
$B_{{\bf R}^D}(0,R)$. Then Theorem \ref{int1} remains valid. This
might be of interest in resonance counting problems, where
self-adjointness of the operator should be preserved in the interior
region where no complex scaling is performed.}
\end{remark}
\begin{remark}\label{int3}
{\rm We believe that the main result of this paper can also be proved in
the case when ${\bf R}^n$ is replaced by a compact manifold. Taking
this for granted, we see that the assumption (\ref{int.6}) cannot be
completely eliminated. Indeed, let $P=hD_x+g(x)$ on ${\bf T}={\bf
  R}/(2\pi {\bf Z})$ where $g$ is smooth and complex valued. Then (cf
Hager \cite{Ha2}) the spectrum of $P$ is contained in the line 
$\Im z = \int_0^{2\pi }\Im g(x)dx/(2\pi )$. This line will vary only very
little under small multiplicative perturbations of $P$ so 
Theorem \ref{int1} cannot hold in this case.}
\end{remark}
\par When $z\in \Sigma \setminus \Sigma _\infty $ and $(\Re z,\Im z)$
is not a critical value of the map $(x,\xi )\to (\Re p,\Im p)$, then
(\ref{int.6.2}) holds with $\kappa =1$. Since the critical values form
a set of Lebesgue measure zero by Sard's theorem, this is what we
expect for most $z$. However such points are necessarily interior points
of $\Sigma $ (by the implicit function theorem) and it is particularly
important to study the distribution of eigenvalues near the
boundary. When $z\in \partial \Sigma \setminus \Sigma _\infty $, and
$\{ p,\{ p,\overline{p}\}\}\ne 0$ at every point of $p^{-1}(z)$, then we
saw in \cite{HaSj}, Example 12.1, that (\ref{int.6.2}) holds with
$\kappa =\frac{3}{4}$.   
\begin{ex}\label{int4}
{\rm Let $1\le m_0(x)$ be an order function on ${\bf R}^n$, let $V\in
S(m_0)$ be a smooth potential which is elliptic in the sense that
$|V(x)|\ge m_0(x)/C$ and assume that $-\pi +\epsilon _0\le
\mathrm{arg\,}(V(x))\le \pi -\epsilon _0$ for some fixed $\epsilon
_0>0$. Then it is easy to see that $p(x,\xi ):=\xi ^2+V(x)$ is an
elliptic element of $S(m)$, where $m(x,\xi )$ is the order function
$m_0(x)+\xi ^2$. Let $\Sigma _{\infty }(V)$ be the set of accumulation
points of $V(x)$ at infinity and define 
$\Sigma (V)=\overline{V({\bf R}^n)}=V({\bf R}^n)\cup \Sigma _\infty
(V)$. Then with $\Sigma $ and $\Sigma _\infty $ defined for $p$ as
above, we get $\Sigma =\Sigma (V)+[0,+\infty [$, $\Sigma_\infty 
=\Sigma_\infty  (V)+[0,+\infty [$. Using the fact that $\partial _{\xi
_1}^2\Re p\ge 1/C$, we further see that if 
$\widetilde{K}\subset {\bf C}$ is compact
and disjoint from $\Sigma _\infty $, then (\ref{int.6.2}) holds
uniformly for $z\in \widetilde{K}$ with $\kappa =1/4$. 
The non-self-adjoint Schr\"odinger
operator $P:=-h^2\Delta +V(x)$ has $P(x,\xi )=p(x,\xi )$ as its symbol
and (\ref{int.6}) is fulfilled. This means that Theorem \ref{int1} is
applicable, but to have an interesting conclusion, 
we have to look for domains $\Gamma $ for which
(\ref{int.7.1}) holds for some $\beta >0$.}
\end{ex}

The conditions on the random perturbations are clearly not the most
general ones attainable with the methods of this paper and further
generalizations may come naturally when looking at new problems. It
should be possible to consider infinite sums in (\ref{int.6.3}) and
drop the upper bound on the size of $\alpha $, provided that we add assumptions on the
probability in (\ref{int.6.5}), (\ref{int.6.6}). Here, we just give an
example where the upper bound $|\alpha |_{{\bf C}^D}\le R$ can be
removed:
Consider
\ekv{int.9}
{
q_\omega (x)=\chi _0(x)\sum_{1}^D \alpha _k(\omega )\epsilon _k(x),
}
as in (\ref{int.6.3}). We now assume that $\alpha _k(\omega )$
are independent Gaussian ${\cal N}(0,\sigma _k^2)$-laws, i.e. with
probability distribution
\ekv{int.10}
{
\frac{1}{\pi \sigma _k^2}e^{-\frac{|\alpha _k|^2}{\sigma
    _k^2}}L(d \alpha _k).
}  
Then $P(d\alpha )$ is of the form (\ref{int.6.5}) (now normalized on
${\bf C}^D$ rather than on the ball $B_{{\bf C}^D}(0,R)$) with
$$
\Phi (\alpha ;h)=-\sum_{k=1}^D \frac{|\alpha _k|^2}{\sigma _k^2}.
$$
On $B_{{\bf C}^D}(0,R)$, we have 
$$
\Vert \nabla \Phi \Vert ={\cal O}(1) \frac{R}{\min \sigma _k^2},
$$
so (\ref{int.6.6}) holds for some $N_4$, provided that $R$ is bounded
by some negative power of $h$ as in (\ref{int.6.3}) and
\ekv{int.11}{\min \sigma _k \hbox{ is bounded from below by some power
  of }h.}

\par As we saw in \cite{HaSj} and further improved and simplified by
Bordeaux Montrieux \cite{Bo}, the probability that $|\alpha |_{{\bf
    C}^D}\ge R$ is 
$$
\le \exp \left( \frac{C_0}{2\min \sigma _j^2}\sum \sigma _j^2 -\frac{R^2}{2\min \sigma _j^2} \right),
$$ 
so
\ekv{int.12}
{
P(|\alpha |_{{\bf C}^D}\ge R)\le e^{-h^{-\widehat{\kappa }}},
}
for $h$ small enough, where $\widehat{\kappa }$ is any given fixed
positive number, provided that $\max \sigma _j$ is bounded from
above by some power of $h$ and we choose $R\asymp h^{-\widetilde{M}}$
for $\widetilde{M}$ large enough. Hence Theorem \ref{int1} is applicable. 

 The remainder of this paper is devoted to the proof of Theorem
 \ref{int1}. Much of the proof follows the strategy of \cite{HaSj} but
 there are also some essential differences, since we had to abandon
 the fairly explicit random matrix considerations there. As in
 \cite{HaSj} we identify the eigenvalues with the zeros of a
 holomorphic function, here $F_\delta (z;h)=\det (P_{\delta ,z})$,
 where $P_{\delta ,z}=(\widetilde{P}_\delta -z)^{-1}(P_\delta
 -z)=1+(\widetilde{P}_\delta -z)^{-1}(P-\widetilde{P})$, 
 $\widetilde{P}_\delta =P_\delta +\widetilde{P}-P$ and $\widetilde{P}$
 is a new pseudodifferential operator, whose symbol coincides with
 the one of $P$ outside a compact set and such that $\widetilde{P}-z$
 is elliptic for all $z\in \Omega $. In {\it Sections \ref{al}, \ref{hs} } we prepare
 this approach by showing that $\delta Q_\omega $ is bounded and has
 small norm: $H^\sigma \to H^\sigma $ for $-s\le \sigma \le s$, where
 $H^\sigma $ is the standard Sobolev space equipped with a natural
 semi-classical $h$-dependent norm). We also need to understand some
 localization and boundedness properties of the resolvent and the
 spectral projections corresponding to small eigenvalues of the
 self-adjoint operators $S_{\delta ,z}=P_{\delta ,z}^*P_{\delta ,z}$
 and $S_\delta =(P_\delta -z)^*(P_\delta -z)$. 

\par In {\it Section \ref{gr}}, we apply results from \cite{HaSj} to estimate
the number of small eigenvalues of $S_{\delta ,z}$ and $S_{\delta
}$. Using this, we set up an auxiliary invertible ``Grushin'' matrix
$$
{\cal P}_\delta =\left(\begin{array}{ccc}P_{\delta ,z} &R_-\\ R_+
    &0 \end{array}\right): L^2({\bf R}^n)\times {\bf C}^N\to 
L^2({\bf R}^n)\times {\bf C}^N,
$$ 
where $N={\cal O}(\alpha ^\kappa h^{-n})$ is the number of eigenvalues
of $S_{\delta ,z}$ that are $\le \alpha $ where $\alpha =Ch$ for some
large constant $C$, and we establish (\ref{grny.9}) saying roughly that 
$$
\ln |\det {\cal P}_\delta |\approx \frac{1}{(2\pi h)^n}\iint \ln
|p_z(x,\xi )|dxd\xi ,\quad p_z=\frac{p-z}{\widetilde{p}-z},
$$
where $\widetilde{p}$ denotes the leading symbol of $\widetilde{P}$.
If $E_{-+}^\delta $ denotes the lower right entry in the block matrix 
of ${\cal P}_\delta^{-1}$ then $\det P_\delta =\det {\cal P}_\delta
+\det E_{-+}^\delta $ as we showed in \cite{HaSj} using some
calculation from \cite{SjZw}. Using that the size $N$ of
$E_{-+}^\delta $ is $\ll h^{-n}$, we get a nice upper bound on 
$\ln|\det E^\delta _{-+}|$ and it follows that for $z$ in a neighborhood of
$\partial \Gamma $,
\begin{equation}\label{int.13}
\ln |F_\delta |\le \frac{1}{(2\pi h)^n}\iint \ln |p_z(x,\xi )|dxd\xi
+\hbox{"small"}.
\end{equation}
See (\ref{sp.31}) for a more precise statement.

The crucial step (as in \cite{Ha,HaSj}) is to get a corresponding
lower bound with probability close to 1 for each $z$, and this amounts
to getting a corresponding lower bound for $\ln |\det E_{-+}^\delta
|$. In \cite{HaSj} we did so by showing that $E_{-+}^\delta $ (there)
was quite close to a random matrix with independent Gaussian
entries. In the case of multiplicative perturbations, such an explicit
approach seems out of reach even if we assume the $\alpha _j$ to be
independent Gaussian random variables. Instead we choose a different
approach based on complex analysis and Jensen's formula in the $\alpha
$-variables. The main step in this new approach is then to construct
one admissible potential as in (\ref{int.6.3}), (\ref{int.6.4}) (ie to find one
special value of $\alpha \in B_{{\bf C}^D}(0,R)$), for which $|\det
E_{-+}^\delta |$ is not too small). When trying to do so, one is led to
consider the singular values of $E_{-+}^\delta $ or
equivalently 
(as we
shall see) the small singular values of $P_{\delta}-z$.

\par In {\it Section \ref{inv}} this is carried out for a model matrix
that would correspond to a leading term in the perturbative expansion of
$E_{-+}^\delta $, however with  $q_\omega $ replaced by a a sum of $N$
delta functions. Then in {\it Section \ref{cl} } we approximate such
$\delta $-functions with admissible potentials and get 
corresponding estimates for a true leading term in the expansion of
$E_{-+}^\delta $. Due to the approximation we only get good lower
bounds for the first roughly $N/2$ singular values.

\par In {\it Section \ref{sv}} we make an iterative procedure. Let
$0<\theta <1/4$ be fixed. and consider the first $\theta N$ 
values of $E_{-+}$ appearing in the inverse of the Grushin matrix for
the unperturbed problem. (For simplicity we here treat $\theta N$ and
similar numbers as if they were integers.) If they are all 
conveniently large, we add no further perturbation in this step, or more precisely we choose the
zero potential as the admissible perturbation. If not, we consider the
perturbation $P_\delta $ given by the special admissible potential $q$
constructed in the preceding section. Then with appropriate choices of
the parameters, we get the desired lower bound on the first $\theta N$
singular values of the matrix $E_{-+}^\delta $, corresponding to this
perturbation. In both cases we get a perturbed operator $P_\delta $ 
(which may or
may not be equal to $P$) and we next consider the natural Grushin
problem for $P_\delta $ now with $N$ replaced by $(1-\theta )N$. For
the new $E_{-+}$ of size $(1-\theta )N$ we again consider the first
$\theta (1-\theta )N $ singular values. If they
are all larger than a new bound, obtained from the preceding one by
multiplication by a
suitable power of $h$, then the next
perturbation is zero, if not, use again the result of the preceding
section to find a convenient perturbation and so on. In the end we get
the desired admissible perturbation as a geometrically convergent sum of
perturbations, and for this perturbation we get 
\begin{equation}\label{int.14}
\ln |F_\delta |\ge (\frac{1}{2\pi h})^n\iint \ln |p_z (x,\xi )|dxd\xi
-\hbox{ "small"}.
\end{equation}

\par In {\it Section \ref{pr}, } the spectral parameter is still
fixed, and we perform a complex analysis argument in
the $\alpha $-variables to show that if we have (\ref{int.14}) for one
value of $\alpha $ then it holds with probability close to 1. In {\it
  Section \ref{en}} it then only remains to let $z $ become
variable and to apply a result of \cite{HaSj} (extending one of
\cite{Ha}) about counting zeros of holomorphic functions with
exponential growth. Very roughly, this result says that if
$u(z)=u(z,\widetilde{h})$ is holomorphic in a fixed neighborhood of
$\overline{\Gamma }$ such that $|u(z;\widetilde{h})|\le e^{\phi
  (z)/\widetilde{h}}$ for all $z$ in a neighborhood of $\partial
\Gamma $ and satisfying the lower bound $|u(z_j;\widetilde{h})|\ge
e^{(\phi (z_j)-\mathrm{small})/h}$ at finitely many points $z_j$, nicely
spread along the boundary of $\Gamma $, then the number of zeros of
$u$ in $\Gamma $ is approximately equal to $(2\pi
\widetilde{h})^{-1}\iint_\Gamma \Delta \phi (z)L(dz)$. Here, as
in \cite{Ha, HaSj} we take $\widetilde{h}=(2\pi h)^n$, $\phi (z)$
equal to the integral in (\ref{int.13}), (\ref{int.14}) and use the
fact that $\frac{1}{2\pi }$ times the Laplacian of this function can
be identified with the push forward under $p$ of the symplectic volume
element.

In {\it Section \ref{app}, } we review some $h$-pseudodifferential and
functional calculus.

\par\noindent {\bf Acknowledgement} The referee's many pertinent
remarks have helped us to improve the presentation of the paper.

\section{Semiclassical Sobolev spaces and multiplication}
\label{al}
\setcounter{equation}{0}
We let $H^s({\bf R}^n)\subset {\cal S}'({\bf R}^n)$, $s\in {\bf R}$, 
denote the semiclassical Sobolev space of order
$s$ equipped with the norm $\Vert \langle hD\rangle^s u\Vert$ where
the norms are the ones in $L^2$, $\ell^2$ or the corresponding
operator norms if nothing else
is indicated. Here $\langle hD\rangle= (1+(hD)^2)^{1/2}$. Let
$\widehat{u}(\xi )=\int e^{-ix\cdot \xi }u(x)dx$ denote the Fourier
transform of the tempered distribution $u$ on ${\bf R}^n$.
\begin{prop}\label{al1}
Let $s>n/2$. Then there exists a constant $C=C(s)$ such that for all
$u,v\in H^s({\bf R}^n)$, we have $u\in L^\infty ({\bf R}^n) $, 
$uv\in H^s({\bf R}^n)$ and 
\ekv{al.1}
{
\Vert u\Vert_{L^\infty }\le Ch^{-n/2}\Vert u\Vert_{H^s},
}
\ekv{al.2}
{
\Vert uv\Vert_{H^s} \le Ch^{-n/2} \Vert u\Vert_{H^s} \Vert v\Vert_{H^s}.
}
\end{prop}
\begin{proof}
The fact that $u\in L^\infty $ and the estimate \no{al.1} follow from
Fourier's inversion formula and the Cauchy-Schwartz inequality:
$$
|u(x)|\le \frac{1}{(2\pi )^n}\int \langle h\xi \rangle^{-s}
(\langle h\xi \rangle^{s}\vert \widehat{u}(\xi )\vert ) d\xi \le 
\frac{1}{(2\pi )^{n/2}}\Vert \langle h\cdot \rangle^{-s}\Vert
\Vert u\Vert_{H^s}.
$$
It then suffices to use that $\Vert \langle h\cdot
\rangle^{-s}\Vert =C(s)h^{-n/2}$. 

\par In order to prove \no{al.2} we pass to the Fourier transform
side, and we see that it suffices to show that
\ekv{al.3}
{
 \int \langle h\xi \rangle^s w(\xi )(\langle h\cdot \rangle^{-s}\widetilde{u}*
\langle h\cdot \rangle^{-s}\widetilde{v})(\xi )d\xi \le
C(s)h^{-\frac{n}{2}}\Vert \widetilde{u}\Vert \Vert \widetilde{v}\Vert \Vert w\Vert,
}
for all non-negative $\widetilde{u},\widetilde{v},w\in L^2$, where $*$ denotes convolution. 
Here the left hand side can be written
$$
\iint_{\eta +\zeta =\xi } \frac{\langle h\xi \rangle^s}{\langle h\eta \rangle^s
\langle h\zeta 
\rangle^s}w(\xi )\widetilde{u}(\eta )\widetilde{v}(\zeta )d\xi d\zeta \le \mathrm{I}+\mathrm{II},
$$
where ${\mathrm I}$, $\mathrm{II}$ denote the corresponding integrals
over the sets $\{ \vert \eta \vert \ge \vert \xi \vert /2\}$ and 
$\{ \vert \zeta \vert \ge \vert \xi \vert /2\}$ respectively. Here
\begin{eqnarray*}
\mathrm{I}&\le& C(s)\int (\int w(\xi )\widetilde{u}(\xi -\zeta )d\xi )\frac{\widetilde{v}(\zeta
  )}{\langle h\zeta \rangle^s}d\zeta \\
&\le& C(s) \Vert w\Vert \Vert \widetilde{u}\Vert \Vert 
\frac{\widetilde{v}}{\langle h\cdot \rangle^s}\Vert_{L^1}.
\end{eqnarray*}
As in the proof of \no{al.1} we see that $\Vert \frac{\widetilde{v}}{\langle
  h\cdot \rangle^s}\Vert_{L^1}\le C(s)h^{-\frac{n}{2}}\Vert
\widetilde{v}\Vert$, so $\mathrm{I}$ is bounded by a constant times $h^{-\frac{n}{2}}
\Vert \widetilde{u}\Vert \Vert \widetilde{v}\Vert \Vert w\Vert$. The same estimate holds for
$\mathrm{II}$ and \no{al.3} follows.
\end{proof}

Let $\widetilde{\Omega}$ be a compact $n$-dimensional manifold. We
cover $\widetilde{\Omega}$ by finitely many coordinate neighborhoods
$M_1,...,M_p$ and for each $M_j$, we let $x_1,...,x_n$ denote the
corresponding local coordinates on $M_j$. Let $0\le \chi _j\in
C_0^\infty (M_j)$ have the property that $\sum_1^p\chi _j >0$ on
$\widetilde{\Omega}$. Define $H^s(\widetilde{\Omega})$ to be the space
of all $u\in {\cal D}'(\widetilde{\Omega})$ such that \ekv{al.4} {
  \Vert u\Vert_{H^s}^2:=\sum_1^p \Vert \chi _j\langle hD\rangle^s \chi
  _j u\Vert ^2 <\infty .  } It is standard to show that this
definition does not depend on the choice of the coordinate
neighborhoods or on $\chi _j$. With different choices of these
quantities we get norms in \no{al.4} which are uniformly equivalent
when $h\to 0$. In fact, this follows from the $h$-pseudodifferential
calculus on manifolds with symbols in the H\"ormander space
$S^m_{1,0}$. (This calculus has been used in several papers like
\cite{IaSjZw, SjZw, WuZw} and for completeness we discuss it in the
appendix, Section \ref{app}.)

An equivalent definition of $H^s(\widetilde{\Omega})$ is the following: Let 
\ekv{al.5}
{
h^2\widetilde{R}=\sum (hD_{x_j})^*r_{j,k}(x)hD_{x_k}
}
be a non-negative elliptic operator with smooth coefficients on $\widetilde{\Omega}$,
where the star indicates that we take the adjoint with respect to
some fixed positive smooth density on $\widetilde{\Omega}$. Then $h^2\widetilde{R}$ is 
essentially
self-adjoint with domain $H^2(\widetilde{\Omega})$, so $(1+h^2\widetilde{R})^{s/2}:L^2\to L^2$ is a
well-defined closed densely defined operator for $s\in {\bf R}$, which
is bounded precisely when $s\le 0$. Standard methods allow to show
that $(1+h^2\widetilde{R})^{s/2}$ is an $h$-pseudodifferential operator with
symbol in $S^s_{1,0}$ and semiclassical principal symbol given by
$(1+r(x,\xi ))^{s/2}$, where $r(x,\xi )=\sum_{j,k}r_{j,k}(x)\xi _j\xi
_k$ is the semiclassical principal symbol of $h^2\widetilde{R}$. 
See Section \ref{app}.
The
$h$-pseudodifferential calculus gives for every $s\in {\bf R}$:
\begin{prop}\label{al2}
  $H^s(\widetilde{\Omega})$ is the space of all $u\in {\cal D}'(\widetilde{\Omega})$ such that 
$(1+h^2\widetilde{R})^{s/2}u\in L^2$ and the norm $\Vert u\Vert_{H^s}$ is
equivalent to $\Vert (1+h^2\widetilde{R})^{s/2}u\Vert$, uniformly when $h\to 0$.
\end{prop}
\begin{remark}\label{al3}
\rm From the first definition we see that Proposition \ref{al1} remains
valid if we replace ${\bf R}^n$ by a compact $n$-dimensional manifold $\widetilde{\Omega}$.
\end{remark}
\section{$H^s$-perturbations and eigenfunctions}\label{hs}
\setcounter{equation}{0}

\par Let $m\ge 1$ be an order function on ${\bf R}^{2n}$
in the sense that 
$$
m(\rho )\le C_0 \langle \rho -\mu \rangle^{N_0}m(\mu ),\ \rho , \mu\in 
{\bf R}^{2n}  
$$
for some fixed positive constants $C_0,N_0$, 
 and let 
$$p\in
S(m):=\{ a\in C^\infty ({\bf R}^{2n});\, \vert \partial _\rho ^\alpha
a(\rho )\vert\le C_\alpha m(\rho ),\, \forall \rho \in {\bf
  R}^{2n},\alpha \in {\bf N}^{2n}\} .$$ 
We assume that $p-z$ is elliptic (in the sense that $(p-z)^{-1}\in S(m^{-1})$) for at least one value
$z\in {\bf C}$. Put $\Sigma =\overline{p({\bf R}^{2n})}=p({\bf
  R}^{2n})\cup \Sigma _\infty $, where $\Sigma _\infty $ is the set
of accumulation values of $p$ near $\rho = \infty $. Let
$p_1,p_2,...\in S(m)$,  
$$P\sim p+hp_1+...
\hbox{ in } S(m),\ h\to 0.$$ Let $\Omega \Subset {\bf C}$ be open simply connected
with $\overline{\Omega }\cap \Sigma _\infty =\emptyset$, $\Omega
\not\subset \Sigma $. Then as in \cite{Ha, HaSj}, we can construct
$\widetilde{p}\in S(m)$, such that
\ekv{hs.1}{
\widetilde{p}=p\hbox{ away from a compact set.}
} 
\ekv{hs.2}
{
\widetilde{p}-z\hbox{ is elliptic in }S(m),\hbox{ uniformly for }z\in
\overline{\Omega }.
}
The construction also shows that $\widetilde{p}$ can be chosen so that
$\widetilde{p}=p$ away from any given neighborhood of
$p^{-1}(\overline{\Omega })$. 

\par
Let $$\widetilde{P}=P+\widetilde{p}-p\sim \widetilde{p}+hp_1+...\in
S(m)$$
 By
$P$, $\widetilde{P}$ we also denote the corresponding $h$-Weyl
quantizations i.e.~the Weyl quantizations of $P(x,h \xi ;h) $ and 
$\widetilde{P}(x,h \xi ;h)$ respectively. (Sometimes it will also be convenient to indicate
the quantization so that if $a$ is a symbol, then $\mathrm{Op\,}(a)$
denotes the corresponding $h$-pseudodifferential operator.) 
Then we know that $(\widetilde{P}-z)^{-1}$ is a
well-defined uniformly bounded operator when $h$ is small, uniformly
for $z\in \overline{\Omega }$, and that $P$ has discrete spectrum in
$\Omega $ which is contained in any given neighborhood of
$\overline{\Omega }\cap \Sigma $ when $h$ is small enough. 

\par We also
recall that the eigenvalues in $\Omega $, counted with their algebraic
multiplicity, coincide with the zeros of the
function $z\mapsto \det (\widetilde{P}-z)^{-1}(P-z)=\det
(1-(\widetilde{P}-z)^{-1}(\widetilde{P}-P))$, counted with their 
multiplicity. In fact, if $z_0\in \Omega $, then its multiplicity
$m(z_0)$ as a zero of the determinant is 
$$
=\mathrm{tr\,}\frac{1}{2\pi i}\int_\gamma  (1+K(z))^{-1}\dot{K}(z)
dz=\mathrm{tr\,} \frac{1}{2\pi i}\int_\gamma
(z-P)^{-1}(z-\widetilde{P})\dot{K}(z) dz,
$$ 
where $\gamma $ is
a small circle centered at $z_0$,
$K(z)=(z-\widetilde{P})^{-1}(\widetilde{P}-P)$,
$\dot{K}(z)=(z-\widetilde{P})^{-1}-(z-\widetilde{P})^{-2}(z-P)$ and the dots
indicate derivatives with respect to $z$, so
$$
m(z_0)=\mathrm{tr\,}\frac{1}{2\pi i}\int_\gamma (z-P)^{-1}dz - 
\mathrm{tr\,}\frac{1}{2\pi i}\int_\gamma (z-P)^{-1}(z-\widetilde{P})^{-1}(z-P)dz.
$$
Here the first term to the right is the rank of the spectral
projection of $P$ at the eigenvalue $z_0$ ie the multiplicity of $z_0$
as an eigenvalue of $P$, and from Lemma 2.2 of \cite{SjVo}, we see
that the second term is equal to 
$$
-\mathrm{tr\,}\frac{1}{2\pi i}\int_\gamma (z-\widetilde{P})^{-1}dz=0.
$$

\par Now, consider the perturbed operator
\ekv{hs.3}
{
P_\delta =P+\delta Q,
}
where $0\le \delta \ll 1$ will depend on $h$ and $Q$ is the operator
of multiplication with $q\in H^s({\bf R}^n)$, satisfying
\ekv{hs.4}{\Vert q\Vert_{H^s}\le h^{\frac{n}{2}}.}
Here $s>n/2$ is fixed and we systematically use the semiclassical
Sobolev spaces in Section \ref{al}. 

\par Put 
\ekv{hs.5}
{
\widetilde{P}_\delta =\widetilde{P}+\delta Q.
}
If 
\ekv{hs.6}
{
\delta \ll 1,\quad h\ll 1,
}
we know from Section \ref{al} that $\| \delta Q\|_{L^2\to L^2}=\delta
\| q\|_{L^\infty }\ll 1$, and hence $(\widetilde{P}_\delta -z)^{-1}$
is a well-defined bounded operator when $h$ is small enough. The
spectrum of $P_\delta $ in $\Omega $ is discrete and coincides with
the zeros of 
$$
\det ((\widetilde{P}_\delta -z)^{-1}(P_\delta -z))=\det
(1-(\widetilde{P}_\delta -z)^{-1}(\widetilde{P}-P)).
$$ 
Notice here that $(\widetilde{P}_\delta -z)^{-1}(\widetilde{P}-P)$ is
a trace class operator and that again the multiplicities of the
eigenvalues of $P_{\delta }$ and of the zeros of the determinant
agree. 
It is also clear that 
$\sigma (P_\delta )\cap \overline{\Omega} $ is
contained in any given neighborhood of $\Sigma \cap \Omega $, when 
$h$ and $\delta $ are sufficiently small.

\par From Section \ref{al} we know that 
$Q={\cal O}(1):H^\sigma \to H^\sigma $ for $\sigma =s$, by
duality we get the same fact when $\sigma =-s$ and finally by
interpolation (or more directly by \no{al.1} applied to $q$) we get it
also for $\sigma =0$. Writing 
\ekv{hs.6.5}
{\widetilde{P}_\delta
  -z=(\widetilde{P}-z)(1+(\widetilde{P}-z)^{-1}\delta Q)=
(1+\delta Q(\widetilde{P}-z)^{-1})(\widetilde{P}-z),
}
and observing that $(\widetilde{P}-z)^{-1}\in \mathrm{Op}(S(\frac{1}{m}))$ is uniformly bounded: 
$H^s\to H^s$, $H^{-s}\to H^{-s}$, when $z\in \overline{\Omega }$, we 
see that 
\ekv{hs.7}
{
(\widetilde{P}_\delta -z)^{-1}={\cal O}(1):H^s\to H^s,\ H^{-s}\to
H^{-s},\ H^0\to H^0,
}
uniformly when $z\in\overline{\Omega }$ and \no{hs.6} holds, and
similarly for $(1+(\widetilde{P}-z)^{-1}\delta Q)^{-1}$, 
$(1+\delta Q(\widetilde{P}-z)^{-1})^{-1}$.

\par Put 
\ekv{hs.8}
{
P_{\delta ,z}:=(\widetilde{P}_\delta -z)^{-1}(P_\delta -z)=
1-(\widetilde{P}_\delta -z)^{-1}(\widetilde{P}-P)=:1-K_{\delta ,z},
}
\ekv{hs.9}
{
S_{\delta ,z}:=P_{\delta ,z}^*P_{\delta ,z}=1-(K_{\delta ,z}+K_{\delta ,z}^*-K_{\delta ,z}^*K_{\delta ,z})=:1-L_{\delta ,z}.
}
Notice that 
\ekv{hs.10}
{K_{\delta ,z},
L_{\delta ,z}={\cal O}(1):H^{-s}\to H^s,
}
when \no{hs.6} holds. For $0\le \alpha \le 1/2$, let $\pi _\alpha
=1_{[0,\alpha ]}(S_{\delta ,z})$ be the spectral projection
corresponding to the spectrum of $S_{\delta ,z}$ in the interval
$[0,\alpha ]$. 

\par We shall study $H^s$ regularization and localization of
$\pi _\alpha $ and of the ana\-logous spectral projections for $(P_\delta
-z)^*(P_\delta -z)$. The reader who is not too much interested in the
technicalities may proceed directly to Proposition \ref{hs1} at the
end of this section.

\par Apply $\pi _\alpha $ to \no{hs.9}:
$$
\pi _\alpha (1-S_{\delta ,z}\pi _\alpha )=L_{\delta ,z}\pi _\alpha .
$$
Here $\| S_{\delta ,z}\pi _\alpha \| \le 1/2$, so $1-S_{\delta ,z}\pi
_\alpha $ is invertible with inverse of norm $\le 2$. It follows
that 
\ekv{hs.11}
{
\pi _\alpha =L_{\delta ,z}\pi _\alpha (1-S_{\delta ,z}\pi _\alpha )^{-1},
}
so under the assumption \no{hs.6}, we see that 
\ekv{hs.12}
{
\pi _\alpha ={\cal O}(1):\, L^2\to H^s,
}
and since $\pi _\alpha =\pi _\alpha \pi _\alpha ^*$ we even get $\pi
_\alpha ={\cal O}(1):\, H^{-s}\to H^s$. 

\par Since $L_{\delta ,z}$ is compact, we know that the range ${\cal
  R}(\pi _\alpha )$ of $\pi _\alpha $ is of finite dimension, $N$. Let 
$e_1,...,e_N$ be an orthonormal basis in this space. An equivalent way
of stating \no{hs.12} is then
\ekv{hs.13}
{
\| \sum_1^N \lambda _je_j\|_{H^s}\le {\cal O}(1)\Vert \lambda
\Vert_{\ell^2},\ \forall \lambda =(\lambda _1,..,\lambda _N)\in{\bf
  C}^N\simeq \ell^2(\{ 1,2,..,N\}).
}

\par If $\chi \in C_b^\infty ({\bf R}^n)=\{ f\in C^\infty ({\bf
  R}^n);\, \partial ^\alpha f \hbox{ is bounded for every }\alpha \in
{\bf N}^n\}$, we have 
$$
[\widetilde{P}_\delta ,\chi ]=[\widetilde{P},\chi ]\in 
h\mathrm{Op\,}(S(m)).
$$
Combining this with \no{hs.6.5} and the fact mentioned right after 
\no{hs.7}, 
we see that 
\ekv{hs.14}
{
(\widetilde{P}_\delta -z)^{-1}[\widetilde{P}_\delta ,\chi ],\ 
[\widetilde{P}_\delta ,\chi ](\widetilde{P}_\delta -z)^{-1}\ ={\cal
  O}(h):
H^{\sigma }\to H^\sigma ,\ \sigma =\pm s,0.
}

\par From this, it is standard to deduce that 
\ekv{hs.15}
{
\chi _1(\widetilde{P}_\delta -z)^{-1}\chi _0={\cal O}(h^\infty ):\ 
H^\sigma \to H^\sigma ,\ \sigma =\pm s,0,
}
if $\chi _1,\chi _0\in C_b^\infty ({\bf R}^n)$ and $\mathrm{dist\,}
(\mathrm{supp\,}\chi _0,\mathrm{supp\,}\chi _1)>0$. In fact, for any
$M\in {\bf N}^*$, choose $\psi _1,...,\psi _M\in C_b ^\infty ({\bf
  R}^n)$, such that $\mathrm{supp\,}\psi _M\cap\mathrm{supp\,}\chi
_1=\emptyset$, $\psi _{j+1}=1$ on $\mathrm{supp\,}\psi _j$, $\psi
_1=1$ on $\mathrm{supp\,}\chi _0$, and use the telescopic formula,
\ekv{hs.15.5}
{\chi _1(\widetilde{P}_\delta -z)^{-1}\chi _0=
\pm \chi _1(\widetilde{P}_\delta -z)^{-1}
[\widetilde{P}_\delta ,\psi _M](\widetilde{P}_\delta -z)^{-1}...
[\widetilde{P}_\delta ,\psi _1](\widetilde{P}_\delta -z)^{-1}\chi _0.
}

\par Let 
\ekv{hs.16}
{
K=\pi _x(\mathrm{supp\,}(\widetilde{p}-p))
}
be the $x$-space projection of $\mathrm{supp\,}(\widetilde{p}-p)$, so
that $K$ is compact. Combining \no{hs.8}, \no{hs.15}, we see that 
\ekv{hs.17}
{
\chi K_{\delta ,z},\, K_{\delta ,z}\chi ={\cal O}(h^\infty ):\
H^\sigma \to H^\sigma ,\ \sigma =\pm s,0,
}
when $\chi \in C_b^\infty ({\bf R}^n)$ satisfies
$\mathrm{supp\,}\chi \cap K=\emptyset $. From \no{hs.9} we get the
same conclusion for $L_{\delta ,z}$ and then we get from \no{hs.11}
that 
\ekv{hs.18}
{
\chi \pi _\alpha ={\cal O}(h^\infty ):\ L^2\to H^s,
}
if $\chi \in C_b^\infty ({\bf R}^n)$, and $\mathrm{supp\,}\chi \cap
K=\emptyset $. Using that $\pi _\alpha =\pi _\alpha ^2$ and that $\pi
_\alpha ={\cal O}(1):H^{-s}\to H^s$, this can be sharpened to the
statement that 
$$
\chi \pi _\alpha ,\ \pi _\alpha \chi ={\cal O}(h^\infty ):\, H^{-s}\to H^s.
$$

\medskip
\par We also need to establish the corresponding results for $P_\delta
-z$. Let 
\ekv{hs.18.1}
{S_\delta =(P_\delta -z)^*(P_\delta -z),\quad 
\widetilde{S}_\delta =(\widetilde{P}_\delta -z)^*(\widetilde{P}_\delta -z),
}
viewed as self-adjoint Friedrichs extensions from
$(\widetilde{P}_\delta -z)^{-1}(H(m))$
with quadratic form domain $H(m)$. 
Then 
$$
S_\delta =\widetilde{S}_\delta +R,
$$
where 
\ekv{hs.18.3}
{R=(P-\widetilde{P})^*(\widetilde{P}_\delta -z)+(\widetilde{P}_\delta -z)^*(P-\widetilde{P})+(P-\widetilde{P})^*(P-\widetilde{P}),}
and we see that 
\ekv{hs.18.4}
{R={\cal O}(1):H^{-s}\to H^s.}
It follows that 
\eekv{hs.18.5}{
(w-S_\delta )^{-1}&=&(w-\widetilde{S}_\delta )^{-1}+ (w-S_\delta
)^{-1}R(w-\widetilde{S}_\delta )^{-1}}{
&=&(w-\widetilde{S}_\delta )^{-1}- (w-\widetilde{S}_\delta
)^{-1}R(w-S_\delta )^{-1}.}

\par If $\widetilde{m}$ is an order function on ${\bf R}^{2n}$, we
define $H(\widetilde{m})$ for $h>0$ small enough, to be the space
$\widetilde{M}^{-1}L^2({\bf R}^n)$, where $\widetilde{M}\in
\mathrm{Op}(S(\widetilde{m}))$ is an elliptic operator, so that
$\widetilde{M}^{-1}\in \mathrm{Op\,}(S(\frac{1}{m}))$.

\begin{remark}\label{hs0} {\rm For future reference we notice that
  $S_\delta $ coincides with $\widehat{S}_\delta :=(P_\delta
  -z)^*(P_\delta -z)$ with domain ${\cal D}(\widehat{S}_\delta )=\{
  u\in H(m);\, (P_\delta -z)u\in H(m)\}$. In fact, $\widehat{S}_\delta
  $ is a closed operator, with domain contained in the
  quadratic form domain $H(m)$ of $S_\delta $, so it suffices to check
  that $\widehat{S}_\delta $ is self-adjoint. Clearly this operator is
  symmetric so it suffices to check that 
$\widehat{S}_\delta ^*\subset \widehat{S}_\delta $. To shorten
notations, assume that $z=0$: If $u\in {\cal D}(\widehat{S}_\delta
^*)$, $\widehat{S}_\delta ^*u=v$, then $(\widehat{S}_\delta \phi
|u)=(\phi |v) $ for all $\phi \in {\cal D}(\widehat{S}_\delta )$, so 
$(P_\delta \phi |P_\delta u)=(\phi |v)={\cal O}(\Vert \phi
\Vert_{H(m)})$, so $(\widetilde{P}_\delta \phi |P_\delta u)={\cal O}(\Vert
\phi \Vert_{H(m)})$, implying that $P_\delta u\in L^2$, since
$\widetilde{P}_\delta :H(m)\to L^2$ is bijective and ${\cal
{\cal D}(\widehat{S}_\delta )}$ is dense in $H(m)$. Using $(P_\delta \phi
|P_\delta u)=(\phi |v)$ again, we get $P^*_\delta P_\delta u=v$ in the
sense of distributions and since $P$ is elliptic near infinity, we
deduce that $u,P_\delta u\in H(m)$, so $u\in {\cal
  D}(\widehat{S}_\delta )$.}
\end{remark}

\par
Let $f\in C_0^\infty(\mathrm{neigh\,}(0,{\bf R}))$ and let
$\widetilde{f}
\in C_0^\infty (\mathrm{neigh\,}(0,{\bf C}))$ be an almost holomorphic
extension. Since $\widetilde{S}_\delta $ has no spectrum in a fixed
neighborhood of $0$, we get (using the Cauchy-Riemann formula
$$
f(S_\delta )=-\frac{1}{\pi }\int \overline{\partial}\widetilde{f}(w)
(w-S_\delta )^{-1}L(dw)) 
$$
 for $f$ supported in that neighborhood,
\eekv{hs.18.7}
{f(S_\delta )&=&- \int \overline{\partial }\widetilde{f}(w) (w-S_\delta
)^{-1}
R(w-\widetilde{S}_\delta )^{-1}\frac{L(dw)}{\pi }}
{
&=& \int \overline{\partial }\widetilde{f}(w) (w-\widetilde{S}_\delta
)^{-1}
R(w-S_\delta )^{-1}\frac{L(dw)}{\pi }
} 
Here, $(w-\widetilde{S}_\delta )^{-1}={\cal O}(1):H^\sigma \to
H^\sigma $, $\sigma =\pm s,0$, so we conclude that 
$$
f(S_\delta )={\cal O}(1):H^{-s}\to
  L^2 \hbox{ and }L^2\to H^s. 
$$
Then 
$f^2(S_\delta )={\cal O}(1):\, H^{-s}\to H^s$. Let $\pi _\alpha
=1_{[0,\alpha ]}(S_\delta )$. It follows that for $0\le \alpha \ll 1$:
\ekv{hs.19}
{
\pi _\alpha ={\cal O}(1):\,H^{-s}\to H^s, 
}
so \no{hs.13} remains valid.
Using the same telescopic formula as above, we shall next show that 
\ekv{hs.20}
{
\chi \pi _\alpha ,\, \pi _\alpha \chi \, ={\cal O}(h^{\infty }):\,
H^{-s}\to H^s,
}
if $\chi \in C_b^{\infty }({\bf R}^n)$ has the property that
$\mathrm{supp\,}(\chi )\cap K=\emptyset$.

\par For $w\in \mathrm{neigh\,}(0)$, we can write
$\widetilde{S}_0-w=\Lambda _1\Lambda _2$, where $\Lambda _j\in
\mathrm{Op\,}(S(m))$ are elliptic. On the other hand (for $\delta \ll 1$), we have 
$$
\widetilde{S}_\delta -w=\widetilde{S}_0-w+(\widetilde{P}-z)^*\delta q
+ \delta \overline{q}(\widetilde{P}-z)+\delta ^2|q|^2.
$$
We get 
$$
\widetilde{S}_\delta -w=\Lambda _1 (1+\underbrace{\Lambda _1^{-1} ((\widetilde{P}-z)^*\delta q
+ \delta \overline{q}(\widetilde{P}-z)+\delta ^2|q|^2)\Lambda
_2^{-1}}_
{={\cal O}(\delta ):\, H^\sigma \to H^\sigma }
)\Lambda _2,
$$
so
$$
(\widetilde{S}_\delta -w)^{-1}=\Lambda _2^{-1}A\Lambda _1^{-1},
$$
where $A={\cal O}(1):\, H^\sigma \to H^\sigma $ and consequently 
\ekv{hs.21}
{
(\widetilde{S}_\delta -w)^{-1}={\cal O}(1):\ H(\frac{\langle \xi
  \rangle^\sigma }{m})\to H(m\langle \xi \rangle^\sigma ).
}

\par Next, consider $(w-S_\delta )^{-1}$ in
\no{hs.18.3}--\no{hs.18.5}. Using \no{hs.21}, we see
that 
$$
(w-S_\delta )^{-1}={\cal O}(\frac{1}{|\Im w|}):\ 
H(\frac{\langle \xi \rangle^\sigma }{m})\to L^2+H(m\langle \xi \rangle
^\sigma ).
$$
Reinjecting this information into the last expression in \no{hs.18.5},
we see that 
\ekv{hs.22}
{
(w-S_\delta )^{-1}={\cal O}(\frac{1}{|\Im w|}):\ H(\frac{\langle \xi
  \rangle^\sigma }{m})\to H(m\langle \xi \rangle^\sigma ).
 }

\par If $\psi \in C_b^{\infty }({\bf R}^n)$ we next see that
\eekv{hs.23}
{
[\widetilde{S}_\delta ,\psi ]&=& [\widetilde{P}^*,\psi
](\widetilde{P}-z+\delta q)+(\widetilde{P}^*-\overline{z}+\delta
\overline{q})[\widetilde{P},\psi ]
}
{
&=& {\cal O}(h):\ H(m\langle \xi \rangle^\sigma )\to
H(\frac{1}{m}\langle \xi \rangle^\sigma ),
} 
and similarly with $\widetilde{S}_\delta $ replaced by $S_\delta
$. We conclude that 
\ekv{hs.24}
{
(w-\widetilde{S}_\delta )^{-1}[\widetilde{S}_\delta ,\psi ]={\cal
  O}(h):
H(m\langle \xi \rangle^\sigma )\to H(m\langle \xi \rangle^\sigma ),
}
\ekv{hs.25}
{
[\widetilde{S}_\delta ,\psi ](w-\widetilde{S}_\delta )^{-1}={\cal
  O}(h):
H(\frac{\langle \xi \rangle^\sigma}{m} )\to H(\frac{\langle \xi \rangle^\sigma}{m} ),
}
and we have the analogous estimates with $\widetilde{S}_\delta $
replaced by $S_\delta $ and ${\cal O}(h)$ replaced by ${\cal O}(h/|\Im
w|)$.

\par Now, let $\chi $ be as in \no{hs.20} and choose $\chi _0\in
C_0^\infty ({\bf R}^n)$ such that $\chi _0=1$ near $K$,
$\mathrm{supp\,}\chi \cap \mathrm{supp\,}(\chi _0)=\emptyset $. Choose
$\psi _1,...,\psi _M$ as in the telescopic formula \no{hs.15.5} with
$\chi _1$ there equal to $\chi $. Then we get
\eeekv{hs.26}
{
&&\chi (w-\widetilde{S}_\delta )^{-1}\chi _0=
}
{
&&\hskip -8truemm\pm \chi (w-\widetilde{S}_\delta )^{-1}[\widetilde{S}_\delta ,\psi _M]
(w-\widetilde{S}_\delta )^{-1}[\widetilde{S}_\delta ,\psi _{M-1}]...
(w-\widetilde{S}_\delta )^{-1}[\widetilde{S}_\delta ,\psi
_1](w-\widetilde{S}_\delta )^{-1}\chi _0
}{&&={\cal O}(h^M):\ H(\frac{\langle \xi \rangle^\sigma }{m})\to
  H(m\langle \xi \rangle^\sigma ).}

\par Write $R=\chi _0R+(1-\chi _0)R$. Here $(1-\chi
_0)(P-\widetilde{P})^*= {\cal O}(h^\infty ):\, H(m_1)\to H(m_2)$ for
all order functions, $m_1,m_2$, so (cf \no{hs.18.3}) 
$$
(1-\chi _0)(P-\widetilde{P})^*(\widetilde{P}_\delta -z),\ (1-\chi
_0)(P-\widetilde{P})^*(P-\widetilde{P})={\cal O}(h^\infty ):\
H^{-s}\to H(m_2).
$$ 
Moreover,
\begin{eqnarray*}
&&(1-\chi _0)(\widetilde{P}_\delta -z)^*(P-\widetilde{P})\\
&=&(1-\chi _0)(\widetilde{P}-z)^*(P-\widetilde{P})+\delta
\overline{q}(1-\chi _0)(P-\widetilde{P})
\\
&=&{\cal O}(h^\infty ):\ H(m_1)\to H^s, 
\end{eqnarray*}
and we conclude that
\ekv{hs.27}
{
(1-\chi _0)R={\cal O}(h^\infty ):\ H^{-s}\to H^s.
}
Combining this with \no{hs.26}, we get 
\ekv{hs.28}
{
\chi (w-\widetilde{S}_\delta )^{-1}R={\cal O}(h^\infty ):H^{-s}\to
H(m\langle \xi \rangle^s ).
}
Using this and \no{hs.22} in the second expression for $f(S_\delta )$
in \no{hs.18.7}, we see that 
\ekv{hs.29}
{
\chi f(S_\delta )={\cal O}(h^\infty ):\ H(\frac{1}{m\langle \xi 
\rangle ^s})\to H(m\langle \xi
\rangle ^s).
}
Choosing $f=1$ on $[0,\alpha ]$, we see that 
$$
\chi \pi _\alpha =\chi f(S_\delta )\pi _\alpha ={\cal O}(h^\infty ):\
H^{-s}\to H(m\langle \xi \rangle^s),
$$
which implies the estimate on $\chi \pi _\alpha $ in
\no{hs.20}, now with $\pi _\alpha =1_{[0,\alpha ]}(S_\delta )$. Passing to the adjoints we get the estimate on $\pi
_\alpha \chi $ and this completes the verification of \no{hs.20}.
\begin{prop}\label{hs1}
Let $P,p,\widetilde{P},\widetilde{p}$ be as in the beginning of this
section. Let $P_\delta $, $\widetilde{P}_\delta $ be given by
(\ref{hs.3}), (\ref{hs.4}), (\ref{hs.5}) (where $s>n/2$ is fixed) and
make the assumption (\ref{hs.6}). Define $P_{\delta ,z}$, $S_{\delta
  ,z}$ as in (\ref{hs.8}), (\ref{hs.9}), and $S_\delta $ as in
(\ref{hs.18.1}) and realize $S_\delta $ as the Friedrichs
extension. Let $\pi _\alpha $ denote either $1_{[0,\alpha ]}(S_{\delta
,z})$ for $0\le \alpha \le 1/2$, or $1_{[0,\alpha ]}(S_\delta )$ for
$0\le \alpha \ll 1$. In both cases, we have $\pi _\alpha ={\cal
  O}(1):H^{-s}\to H^s$ uniformly with respect to $\alpha
,h$, implying (\ref{hs.13}). 
Moreover, if $\chi \in C_b^\infty ({\bf R}^n)$ is independent of
$h$ and $\mathrm{supp\,}\chi \cap \pi
_x(\mathrm{supp\,}(\widetilde{p}-p))=\emptyset $ (cf (\ref{hs.16})),
then $\chi \pi _\alpha $, $\pi _\alpha \chi $ are $={\cal O}(h^\infty
):H^{-s}\to H^s$. In the second case we also have $\chi \pi _\alpha
={\cal O}(h^\infty ):H^{-s}\to H(m\langle \xi \rangle ^s)
$.
\end{prop}

\section{Grushin problems}\label{gr} \setcounter{equation}{0} 
Let $P:{\cal H}\to {\cal H}$ be a bounded operator, where ${\cal H}$
is a complex separable Hilbert space. Following the standard
definitions (see \cite{GoKr}) we define the singular values of $P$ to
be the decreasing sequence $s_1(P)\ge s_2(P)\ge ...$ of eigenvalues of
the self-adjoint operator $(P^*P)^{1/2}$ as long as these eigenvalues
lie above the supremum of the essential spectrum. If there are only
finitely many such eigenvalues, $s_1(P),...,s_k(P)$ then we define
$s_{k+1}(P)=s_{k+2}(P)=...$ to be the supremum of the essential
spectrum of $(P^*P)^{1/2}$. When $\mathrm{dim\,}{\cal H}=M<\infty $
our sequence is finite (by definition); $s_1\ge s_2\ge ...\ge s_M$,
otherwise it is infinite. Using that if $P^*Pu=s_j^2u$, then
$PP^*(Pu)=s_j^2Pu$ and similarly with $P$ and $P^*$ permuted, we see
that $s_j(P^*)=s_j(P)$. Strictly speaking, $P^*P:\,{\cal N}(P)^\perp
\to {\cal N}(P)^\perp$ and $PP^*:\,{\cal N}(P^*)^\perp
\to {\cal N}(P^*)^\perp$ are unitarily equivalent via the map
$P(P^*P)^{-1/2}:\, {\cal N}(P)^\perp \to {\cal N}(P^*)^\perp$ and its
inverse
$P^*(PP^*)^{-1/2}:\, {\cal N}(P^*)^\perp \to {\cal N}(P)^\perp$. (To
check this, notice that the relation $P(P^*P)=(PP^*)P$ on ${\cal
  N}(P)^\perp$ implies $P(P^*P)^\alpha =(PP^*)^\alpha P$ on the same
space for every $\alpha \in {\bf R}$.)

\par In the case when $P$ is a Fredholm operator of index $0$, it will
be convenient to introduce the increasing sequence $0\le t_1(P)\le
t_2(P)\le ...$ consisting first of all eigenvalues of $(P^*P)^{1/2}$
below the infimum of the essential spectrum and then, if there are
only finitely many such eigenvalues, we repeat indefinitely that
infimum. (The length of the resulting sequence is the
dimension of ${\cal H}$.) When $\mathrm{dim\,}{\cal H}=M<\infty $, we
have $t_j(P)=s_{M+1-j}(P)$. Again, we have $t_j(P^*)=t_j(P)$ (as
reviewed in \cite{HaSj}). Moreover, in the case when $P$ has a 
bounded inverse, we see that 
\ekv{}{s_j(P^{-1})=\frac{1}{t_j(P)}.}

\par Let $P$ be a Fredholm operator of index $0$. Let $1\le N<\infty $
and let $R_+:{\cal H}\to {\bf C}^N$, $R_-:{\bf C}^N\to {\cal H}$ be
bounded operators. Assume that 
\ekv{gr.01}{
{\cal P}=\left(\begin{array}{ccc}P &R_-\\ R_+ &0 \end{array}\right):
{\cal H}\times {\bf C}^N\to {\cal H}\times {\bf C}^N
}
is bijective with a bounded inverse
\ekv{gr.02}
{
{\cal E}=\left(\begin{array}{ccc}E &E_+\\E_- &E_{-+} \end{array}\right)
}

\par
Recall (for instance from \cite{SjZw2}) that $P$ has a bounded inverse precisely when 
$E_{-+}$ has, and when this happens we have the relations,
\ekv{s.8}{
P^{-1}=E-E_+E_{-+}^{-1}E_-,\quad E_{-+}^{-1}=-R_+P^{-1}R_-.
}
Recall (\cite{GoKr}) that if $A,B$ are bounded operators, then we have
the general estimates,
\ekv{s.9}
{
s_{n+k-1}(A+B)\le s_n(A)+s_k(B),
}
\ekv{s.10}
{
s_{n+k-1}(AB)\le s_n(A)s_k(B),
}
in particular for $k=1$, we get
$$
s_n(AB)\le \Vert A\Vert s_n(B),\ s_n(AB)\le s_n(A)\Vert B\Vert,\
s_n(A+B)\le s_n(A)+\Vert B\Vert . 
$$ 
Applying this to the second part of \no{s.8}, we get 
$$
s_k(E_{-+}^{-1})\le \Vert R_-\Vert \Vert R_+\Vert s_k(P^{-1}),\ 1\le
k\le N
$$
implying
\ekv{s.11}
{
t_k(P)\le \Vert R_-\Vert \Vert R_+\Vert t_k(E_{-+}),\ 1\le k\le N.
}
By a perturbation argument, we see that this holds also in the
case when $P$, $E_{-+}$ are non-invertible. 

\par Similarly from the first part of \no{s.8}, we get 
$$
s_k(P^{-1})\le \Vert E\Vert +\Vert E_+\Vert \Vert E_-\Vert s_k(E_{-+}^{-1}),
$$
leading to
\ekv{s.12}
{
t_k(P)\ge \frac{t_k(E_{-+})}{\| E\| t_k(E_{-+})+\Vert E_+\Vert\Vert E_-\Vert}.
}
Again this can be extended to the non-necessarily invertible case by
means of small perturbations.

\par Next, we recall from \cite{HaSj} a natural construction of an
associated Grushin problem to a given operator. Let $P_0:{\cal H}\to
{\cal H}$ be a Fredholm operator of index $0$ as above. Assume that
the first $N$ singular values $t_1(P_0)\le t_2(P_0)\le ...\le t_N(P_0)$
correspond to discrete eigenvalues of $P_0^*P_0$ and assume that $t_{N+1}(P_0)$ is
strictly positive. In the following we sometimes write $t_j$ instead
of $t_j(P_0)$ for short.

\par Recall that $t_j^2$ are the first
eigenvalues both for $P_0^*P_0$ and $P_0P_0^*$.
%
Let $e_1,...,e_N$ and
$f_1,...,f_N$ be corresponding orthonormal systems of eigenvectors
of $P_0^*P_0$ and $P_0P_0^*$ respectively. They can be chosen so that
\ekv{gr.1} { P_0e_j=t _jf_j,\ P_0^*f_j=t_je_j.
} Define $R_+:L^2\to {\bf C}^N$ and $R_-:{\bf C}^N\to L^2$ by
\ekv{gr.2}{R_+u(j)=(u|e_j),\ R_-u_-=\sum_1^Nu_-(j)f_j.}  
As in \cite{HaSj}, the Grushin problem \ekv{gr.3} {
\left\{
\begin{array}{ll}P_0u+R_-u_-=v,\\ R_+u=v_+,
 \end{array} \right.  } has a unique solution $(u,u_-)\in L^2\times
{\bf C}^N$ for every $(v,v_+)\in L^2\times {\bf C} ^N$, given by \ekv{gr.4} {
\left\{\begin{array}{ll} u=E^0v+E_+^0v_+,\\ u_-=E_-^0v+E_{-+}^0v_+,
\end{array}\right.}  where
\begin{eqnarray}\label{gr.4.5} E^0_+v_+=\sum_1^Nv_+(j)e_j,& E^0_-v(j)=(v|f_j),\\
E^0_{-+}=-{\rm diag\,}(t _j),& \Vert E^0\Vert \le
{1\over t_{N+1}}.\nonumber
\end{eqnarray}
$E^0$ can be viewed
as the inverse of $P_0$ as an operator from the orthogonal space
$(e_1,e_2,...,e_N)^\perp$ to $(f_1,f_2,...,f_N)^\perp$.
\par
We notice that in this case, the norms of $R_+$ and $R_-$ are equal to
1, so \no{s.11} tells us that $t_k(P_0)\le t_k(E^0_{-+})$ for $1\le k\le
N$, but of course the expression for $E^0_{-+}$ in \no{gr.4.5} implies
equality. 

\par Let $Q\in {\cal L}({\cal H}, {\cal H}) $ and put $P_\delta
=P_0-\delta Q$ (where we sometimes put a minus sign in front of the
perturbation for notational convenience). We are particularly
interested in the case when $Q=Q_\omega u=q_\omega u$ is the operator of
multiplication with a random function $q_\omega $. Here $\delta >0$
is a small parameter. Choose $R_\pm$ as in \no{gr.2}. Then if $\delta
<t_{N+1}$ and $\Vert Q\Vert\le 1$, 
the perturbed Grushin problem 
\ekv{gr.5}{ \left\{
\begin{array}{ll}P_\delta u+R_-u_-=v,\\ R_+u=v_+,
 \end{array} \right.  } is well posed and has the solution 
\ekv{gr.6}
{ \left\{\begin{array}{ll} u=E^\delta v+E_+^\delta v_+,\\
u_-=E_-^\delta +E_{-+}^\delta v_+,
\end{array} \right.}  where
\ekv{gr.6.5}
{
{\cal E}^\delta =\left(\begin{array}{ccc}E^\delta &E_+^\delta \\
E_-^\delta &E_{-+}^\delta
 \end{array}\right)
}
is obtained from ${\cal E}^0$ by
\ekv{gr.6.6}{
{\cal E}^\delta ={\cal E}^0\left( 1-\delta
\left(\begin{array}{ccc}Q E^0 & Q E_+^0 \\
0&0 \end{array}\right)\right) ^{-1}.
}
Using the Neumann series, we get
\ekv{gr.7} {
E_{-+}^\delta =E_{-+}^0+\delta E_-^0 Q E_+^0+ \delta^2
E_-^0 Q E^0 Q E_+^0+ \delta^3 E_-^0 Q
(E^0 Q )^2 E_+^0+...  }
We also get 
\ekv{gr.7.1}
{
E^\delta =E^0+\sum_1^{\infty }\delta ^kE^0(QE^0)^k
}
\ekv{gr.7.2}
{
E_+^\delta =E_+^0+\sum_1^{\infty }\delta ^k(E^0Q)^kE_+^0
}

\ekv{gr.7.3}
{
E_-^\delta =E_-^0+\sum_1^{\infty }\delta ^kE_-^0(QE^0)^k.
}
\par The leading perturbation in $E_{-+}^\delta $ is
$\delta M $, where $M =E_-^0 Q E_+^0: {\bf C}^N\to
{\bf C}^N$ has the matrix \ekv{gr.8}{ M(\omega )_{j,k}=(Q e_k|f_j),
} which in the multiplicative case reduces to \ekv{gr.9} { M(\omega
)_{j,k}=\int q (x)e_k(x)\overline{f_j(x)}dx.  }

\par Put $\tau _0=t_{N+1}(P_0)$ and recall the 
assumption
\ekv{s.13}
{
\Vert Q\Vert \le 1.
}
Then, if $\delta \le \tau _0/2$, the new Grushin problem is well posed
with an inverse ${\cal E}^{\delta }$ given in
\no{gr.6.5}--\no{gr.7.3}. We get 
\ekv{s.14}
{
\Vert E^\delta \Vert \le \frac{1}{1-\frac{\delta }{\tau _0}}
\Vert E^0\Vert \le \frac{2}{\tau _0},\quad
\| E_\pm ^\delta \| \le \frac{1}{1-\frac{\delta }{\tau _0}}\le 2,  
}
\ekv{s.15}
{
\Vert E_{-+}^\delta -(E_{-+}^0+\delta E_-^0QE_+^0)\Vert \le
\frac{\delta ^2}{\tau _0}\frac{1}{1-\frac{\delta }{\tau_0}}\le 
2\frac{\delta ^2}{\tau _0}.
}
Using this in \no{s.11}, \no{s.12} together with the fact that $t_k(E_{-+}^\delta )
\le 2\tau_0$, we get
\ekv{s.16}
{
\frac{t_k(E^\delta _{-+})}{8}\le t_k(P_\delta )\le t_k(E^\delta _{-+}).
}

\begin{remark}\label{gr0} \rm under suitable assumptions, the 
preceding discussion can be extended to the case of unbounded
operators. The purpose of this remark is to make one such extension
that will be needed later. Let $P$, $m$, $P_{{\delta_0}}$ be as in
Section \ref{hs}, satisfying (\ref{hs.4}), \no{hs.6} with $\delta $ there equal to
$\delta _0$. We fix $z\in \Omega $ with $\Omega $, $\Sigma $, $\Sigma
_\infty $ as in that section. For notational convenience, we may
assume that $z=0$. Then we know that $P_{\delta_0} :\, H(m)\to L^2({\bf
  R}^n)$ is Fredholm of index $0$, and the same holds for the formal
adjoint $P_{\delta_0} ^*$.

\par Let 
\ekv{gr.11}{
S_{\delta_0} =P_{\delta_0} ^*P_{\delta_0} ,\quad T_{\delta_0} =P_{\delta_0} P_{\delta_0} ^*
}
be the unbounded operators equipped with their natural domains,
\ekv{gr.12}
{
{\cal D}(S_{\delta_0} )=\{ u\in L^2;\, P_{\delta_0} u\in L^2,\ P_{\delta_0}
^*(P_{\delta_0} u)\in L^2\} =\{ u\in H(m);\, P_{\delta _0}u\in H(m)\}  ,
} 
and similarly for $T_{\delta_0} $. From Remark \ref{hs0} we know that
$S_{\delta _0}$ is self-adjoint and we clearly have the same fact for
$T_{\delta _0}$.

\par It is now easy to check that $S_{\delta_0} \ge 0$, $T_{\delta_0} \ge 0$
have discrete spectra in a fixed neighborhood of $0$, using that
$S_{\delta _0}-\widetilde{S}_{\delta _0}$ and $T_{\delta
  _0}-\widetilde{T}_{\delta _0}$ are compact, where
$\widetilde{S}_{\delta _0}$ and $\widetilde{T}_{\delta _0}$ are
defined as in \no{gr.11} with $P_{\delta _0}$ replaced by
$\widetilde{P}_{\delta _0}$. Moreover, 
$$
{\cal N}(S_{\delta_0} )=\{ u\in H(m);\, P_{\delta_0} u=0\},\ 
{\cal N}(T_{\delta_0} )=\{ u\in H(m);\, P^*_{\delta_0} u=0\} ,
$$
and since $P_{\delta_0} $, $P_{\delta_0} ^*$ are Fredholm of index 0, we
deduce that 
\ekv{gr.14}
{
\mathrm{dim\,}{\cal N}(S_{\delta_0} )=\mathrm{dim\,}{\cal N}(T_{\delta_0} ).}

\par Further, if $S_{\delta_0} u=\lambda u$, $\Vert u\Vert =1$, $0<\lambda
\ll 1$, then we can apply $P_{\delta_0} $ and write
\ekv{gr.15}
{
P_{\delta_0} P_{\delta_0} ^*(P_{\delta_0} u)=\lambda (P_{\delta_0} u).
}
Here $P_{\delta _0}u\in H(m)$ (the quadratic form domain of $T_{\delta
_0}$) and since the right hand side is (a fortiori) in $L^2$, we see
that $P_{\delta _0}u\in {\cal D}(T_{\delta _0})$
 and that $T_{\delta_0} (P_{\delta_0}
u)=\lambda (P_{\delta_0} u)$. Similarly, if $T_{\delta_0} v=\lambda v$, $\Vert
v\Vert =1$, $0<\lambda \ll 1$, we see that 
$P^*_{\delta_0} v\in {\cal D}(S_{\delta_0} )$ and that $S_{\delta_0} (P_{\delta_0}^*
v)=\lambda (P_{\delta_0}^* v)$.

\par It is then clear that if $0<\alpha \ll 1$, then $S_{\delta_0}
,T_{\delta_0} $ have the same eigenvalues in $[0,\alpha ]$, and if
these eigenvalues are denoted by $0\le t_1^2\le t_2^2\le ...\le
t_N^2\le \alpha $ with $t_j\ge 0$, then we can find orthonormal families of 
eigenfunctions, $e_1,e_2,...,e_N\in {\cal D}(S_{\delta_0} )$, 
$f_1,f_2,...,f_N\in {\cal D}(T_{\delta_0} )$, such that
\ekv{gr.16}
{
P_{\delta_0} e_j=t_jf_j,\quad P_{\delta_0} ^*f_j=t_je_j,
}
in analogy with \no{gr.1}

\par From this point on, the discussion from \no{gr.1} to \no{s.16}
goes through with only minor changes, with $P_0$ replaced by
$P_{\delta _0}$ and $P_\delta $ replaced by a new perturbation
$P_{\delta _0}+\delta Q_{\mathrm{new}}$. End of the remark.
\end{remark}

\par We next collect some facts from \cite{HaSj}. The first result
follows from Section 2 in that paper.
\begin{prop}\label{gr1}
Let $P:{\cal H}\to {\cal H}$ be bounded and assume that $P-1$ is of
trace class, so that $P$ is Fredholm of index $0$. Let $R_+,R_-,{\cal
  P}, {\cal E}={\cal P}^{-1}$ be as in \no{gr.01}, \no{gr.02}. Then
${\cal P}$ is also a trace class perturbation of the identity operator
and 
\ekv{grny.1}{\det P=\det {\cal P}\det E_{-+}.
} 
\end{prop}

\par Now consider the operator $P_z=P_{0,z}$ in \no{hs.8} 
for $z\in \Omega $, and
recall that $P_z$ is a trace class perturbation of the identity. Put
$s(x,\xi )=s_z(x,\xi )=| p_z(x,\xi )|^2$. Following
Section 4 in \cite{HaSj}, we introduce $V(t)=V_z(t)$ by
\ekv{grny.2}
{
V(t)=\iint_{s(x,\xi )\le t}dxd\xi ,\ 0\le t\le \frac{1}{2}.
}
For a given $z\in \Omega $, we assume that there exists $\kappa \in
]0,1]$, such that 
\ekv{grny.3}
{
V(t)={\cal O}(1)t^\kappa ,\ 0\le t\le \frac{1}{2}
}
(Later on we shall also assume that this condition holds uniformly
when $z$ varies in some subset of $\Omega $, and then all estimates
below will hold uniformly for $z$ in that subset.) Proposition 4.5 in
\cite{HaSj} and a subsequent remark there give
\begin{prop}\label{gr2}
Assume \no{grny.3}. For $0<h\ll \alpha \ll 1$, the number $N(\alpha )$
of eigenvalues of $P_z^*P_z$ in $[0,\alpha ]$ satisfies
\ekv{grny.4}
{
N(\alpha )={\cal O}(\alpha ^\kappa h^{-n}).
}
Moreover, 
\ekv{grny.5}
{
\ln \det P_z^*P_z\le \frac{1}{(2\pi h)^n}(\iint \ln (s)dxd\xi +
{\cal O}(\alpha ^\kappa \ln \frac{1}{\alpha })).
}
\end{prop}

We next consider $P_{\delta ,z}=(\widetilde{P}_\delta
-z)^{-1}(P_\delta -z)=1-K_{\delta ,z}$ with $P_\delta $,
$\widetilde{P}_\delta $ as in Section \ref{hs} and under the
assumptions \no{hs.4}, \no{hs.6}. Put 
$$
S_{\delta,z} =P^*_{\delta ,z}P_{\delta ,z}=1-K_{\delta ,z}-K_{\delta ,z}^*+K_{\delta ,z}^*K_{\delta ,z},
$$
where $K_{\delta ,z}$ is given by \no{hs.8}, so that
$$
\Vert K_{\delta ,z}\Vert \le {\cal O}(1),\ \Vert K_{\delta
  ,z}\Vert_{\mathrm{tr}}
\le \Vert (\widetilde{P}_\delta -z)^{-1}\Vert \Vert
\widetilde{P}-P\Vert_{\mathrm{tr}}\le {\cal O}(h^{-n}).
$$
Here $\Vert \cdot \Vert_{\mathrm{tr}}$ denotes the trace class norm,
and we refer for instance to \cite{DiSj} for the standard estimate on
the trace class norm of an $h$-pseudodifferential operator with
compactly supported symbol, that we used for the last estimate.

\par Write $\dot{K}_{\delta ,z}=\frac{\partial }{\partial \delta }K_{\delta
,z}$. Then 
$$
\dot {K}_{\delta ,z}=-(z-\widetilde{P}_\delta
)^{-1}Q(z-\widetilde{P}_\delta )^{-1}(\widetilde{P}-P),
$$
so 
$$
\Vert \dot{K}_{\delta ,z}\Vert \le {\cal O}(\Vert Q\Vert ),\quad \Vert \dot{K}_{\delta ,z}\Vert_{\mathrm{tr}} \le {\cal O}(\Vert Q\Vert h^{-n}).
$$
It follows that 
$$
\Vert \dot{S}_{\delta, z}\Vert \le {\cal O}(\Vert Q\Vert ),\quad \Vert
\dot{S}_{\delta, z}\Vert_{\mathrm{tr}} \le {\cal O}(\Vert Q\Vert h^{-n}).
$$

\par Let $N=N(\alpha ,\delta )$ denote the number of singular values
of $P_{\delta ,z}$ in the interval $[0,\sqrt{\alpha} [$ for $h\ll \alpha \ll
1$. Strengthen the assumption \no{hs.6} to 
\ekv{grny.5.5}
{
\delta \le {\cal O}(h).
}
Then $\Vert S_{\delta ,z}-S_{0,z} \Vert \le {\cal O}(h)$ and from
\no{grny.4} we get 
\ekv{grny.5.6}
{
N(\alpha ,\delta )={\cal O}(\alpha ^\kappa h^{-n}).
}
 Define 
$$
{\cal P}_\delta =\left(\begin{array}{ccc}P_{\delta ,z} &R_{-,\delta
    }\\ R_{+,\delta } &0
 \end{array}\right)
$$
as in \no{gr.1}--\no{gr.3}, so that ${\cal P}={\cal P}_0$. As in (5.10)
in \cite{HaSj} we have 
\ekv{grny.6}
{
|\det {\cal P}_\delta |^2=\alpha ^{-N}\det 1_\alpha (S_{\delta,z} ),\quad 2\ln |\det
{\cal P}_\delta |=\ln \det 1_\alpha (S_{\delta,z} )+N\ln
\frac{1}{\alpha }, 
}
where $1_\alpha (t)=\max (\alpha ,t)$, $t\ge 0$. (The different power of
$\alpha $ is due to the 
normalizing factor $\sqrt{\alpha }$, used in the definition of 
$R_\pm $ in \cite{HaSj}.)

\par For $0<\epsilon \ll 1 $, let $C^\infty (\overline{{\bf R}}_+)\ni
1_{\alpha ,\epsilon }\ge 1_\alpha $ be equal to $t$ outside a small
neighborhood of $t=0$ and converge to $1_\alpha $ uniformly when
$\epsilon \to 0$. For any fixed $\epsilon >0 $, we put $f(t)=1_{\alpha
,\epsilon }(t)$ for $t\ge 0$ and extend $f$ to ${\bf R}$ in
such a way that $f(t)=t+g(t)$, $g\in C_0^\infty ({\bf R})$.  Let
$\widetilde{f}(t)=t+\widetilde{g}(t)$ be an almost holomorphic
extension of $f$ with $\widetilde{g}\in C_0^\infty ({\bf C})$.  Then
we have the Cauchy-Riemann formula (see for instance \cite{DiSj} and
further references given there):
$$
f(S_{\delta,z} )=S_\delta -\frac{1}{\pi }\int (w-S_{\delta,z}
)^{-1}\overline{\partial }\widetilde{g}(w)L(dw).
$$  
From this we see that
$$
\frac{\partial }{\partial \delta }f(S_{\delta,z} )=\dot{S}_\delta
-\frac{1}{\pi } \int (w-S_{\delta,z} )^{-1}\dot{S}_{\delta ,z}
(w-S_{\delta,z} )^{-1}\overline{\partial }\widetilde{g}(w)L(dw).
$$ Now,
\begin{eqnarray*} &&\hskip -1truecm\frac{\partial }{\partial \delta
}\ln\det f(S_{\delta,z} )= \mathrm{tr\,} f(S_{\delta,z}
)^{-1}\frac{\partial }{\partial \delta }f(S_{\delta,z} )=\\ &&\hskip
-1truecm\mathrm{tr\,}(f(S_\delta )^{-1}\dot{S}_\delta ) -\frac{1}{\pi
}\int \mathrm{tr\,}(f(S_{\delta,z} )^{-1}(w-S_{\delta,z}
)^{-1}\dot{S}_{\delta ,z} (w-S_{\delta,z} )^{-1}) \overline{\partial
}\widetilde{g}(w)L(dw).
\end{eqnarray*}
Here $f(S_{\delta,z} )^{-1}$ and $(w-S_{\delta,z} )^{-1}$ commute, and using also the
cyclicity of the trace, we see that the last term is equal to
\begin{eqnarray*}
&&\mathrm{tr\,}(f(S_{\delta,z} )^{-1}\frac{(-1)}{\pi }\int (w-S_{\delta,z}
)^{-2}
\overline{\partial }_w\widetilde{g}(w) L(dw)\dot
{S}_{\delta ,z} )\\ &=&\mathrm{tr\,}(f(S_{\delta,z} )^{-1}\frac{(-1)}{\pi }\int
(w-S_{\delta,z} )^{-1} \overline{\partial }_w\partial
_w\widetilde{g}(w)L(dw)\dot{S}_{\delta ,z} )\\
&=&\mathrm{tr\,}(f(S_{\delta ,z})^{-1}g'(S_{\delta ,z})\dot{S}_{\delta
,z}),
\end{eqnarray*}
leading to the general identity
$$
\frac{\partial }{\partial \delta }\ln \det f(S_{\delta,z} )=\mathrm{tr\,}(
f(S_{\delta,z} )^{-1}f'(S_{\delta,z} )\dot {S}_{\delta ,z} ).
$$
Now we can choose $f=1_{\alpha,\epsilon }$ such that $|f'(t)|\le 1$
for $t\ge 0$. Then we get the estimate
\begin{eqnarray*}
\frac{\partial }{\partial \delta }\ln \det (1_{\alpha ,\epsilon }(S_{\delta,z}
))&=&\mathrm{tr\,}(1_{\alpha ,\epsilon }(S_{\delta,z} )^{-1}1'_{\alpha
  ,\epsilon }(S_{\delta,z} )\dot {S}_{\delta ,z} )\\
&=& {\cal O}(\frac{\Vert \dot{S}_{\delta ,z} \Vert _{\mathrm{tr}}}{\alpha
})\\
&=& {\cal O}(1)\frac{\Vert Q\Vert}{\alpha h^n}.
\end{eqnarray*}

\par Since $\ln \det 1_\alpha (S_{\delta,z} )=\lim_{\epsilon \to 0}\ln\det
1_{\alpha ,\epsilon }(S_{\delta,z} )$, we can integrate the above
estimate, pass to the limit and obtain
$$
\ln \det 1_\alpha (S_{\delta ,z})=\ln\det 1_{\alpha }(S_{0,z} )+{\cal
  O}(\frac{\delta \Vert Q\Vert}{\alpha h^n}).
$$
Using \no{grny.6}, \no{grny.5.6}, we get
\ekv{grny.7}
{
\ln |\det {\cal P}_\delta |^2= \ln |\det {\cal P}|^2+{\cal
  O}(\frac{\delta \Vert Q\Vert}{\alpha h^n}+\alpha ^\kappa
h^{-n}\ln\frac{1}{\alpha })
.
}
The estimate (5.13) in \cite{HaSj} is valid in our case:
\ekv{grny.8}
{
\ln |\det {\cal P}|=\frac{1}{(2\pi h)^n}(\iint \ln |p_z| dxd\xi +{\cal
  O} (\alpha ^\kappa \ln \frac{1}{\alpha })),
}
and using this in \no{grny.7},we get
\ekv{grny.9}{
\ln |\det {\cal P}_\delta |=\frac{1}{(2\pi h)^n}(\iint \ln |p_z| dxd\xi 
+{\cal O}(\alpha ^\kappa \ln \frac{1}{\alpha }+\frac{\delta }{\alpha
}\Vert Q\Vert )). 
}

\section{Singular values and determinants of certain
matrices associated to $\delta $ potentials}\label{inv} \setcounter{equation}{0} 
We start with a general
observation.
\begin{prop}\label{inv1} If $e_1(x),...,e_N(x)$ are linearly independent
continuous functions on an open domain $\Omega \subset {\bf R}^n $, 
then we can find $N$
different points $a_1,...,a_N\in\Omega $ so that $\overrightarrow{e}(a_1),...,\overrightarrow{e}(a_N)$
are linearly independent in ${\bf C}^N$, where
$$
\overrightarrow{e}(x)=\pmatrix{e_1(x)\cr e_2(x)\cr ..\cr ..\cr e_N(x)}.
$$
\end{prop}
\begin{proof} Let $E\subset {\bf C}^N$ be the linear subspace
spanned by all the $\overrightarrow{e}(x)$, $x\in \Omega $. We claim that
$E={\bf C}^N$. Indeed, if that were not the case, there would exist
$0\ne (\lambda _1,...,\lambda _N)\in{\bf C}^N$ such that
$$
0=\langle \lambda , \overrightarrow{e}(x)\rangle :=\sum_1^N\lambda _je_j(x),\quad
\forall x\in \Omega .
$$
But this means that $e_1,...,e_N$ are \emph{ linearly dependent }
functions in contradiction with the assumption, hence $E={\bf C}^N$ and
then we can find $a_1,..,a_N\in\Omega $ such that $\overrightarrow{e}(a_1),...,\overrightarrow{e}(a_N)$
form a basis in ${\bf C}^N$ and consequently so that they are linearly
independent.
\end{proof}
\begin{prop}\label{inv2} Let $e_1,...,e_N$ be as in Proposition
\ref{inv1} and let $f_1,...,f_N$ be a second family with the same
properties. Assume that we can find $a_1,...,a_N\in\Omega $ such that
both $\{ \overrightarrow{e}(a_1),...,\overrightarrow{e}(a_N)\}$ and $\{ \overrightarrow{f}(a_1),...,\overrightarrow{f}(a_N)\}$ are
linearly independent. (We notice that this holds in the special case
when $f_j=\overline{e}_j$.) Define $M={\bf C}^N\to {\bf C}^N$ by \ekv{inv.0} {
Mu=\sum_1^N (u|\overrightarrow{f}(a_\nu ))\overrightarrow{e}(a_\nu ),\
u\in{\bf C}^N,  }
where $(\cdot |\cdot \cdot )$ denotes the usual scalar product on
${\bf C}^N$.
 Then $M$ is
bijective.
\end{prop}
\begin{proof} Let $u\in{\bf C}^N$ belong to the kernel of $M$. Since
$\overrightarrow{e}(a_1),...,\overrightarrow{e}(a_N)$ form a basis in ${\bf C}^N$, we have $(u|\overrightarrow{f}(a_\nu ))=0$
for all $\nu $. Since $\overrightarrow{f}(a_1),...,\overrightarrow{f}(a_N)$ form a basis in ${\bf C}^N$, it
then follows that $u=0$.
\end{proof}
\begin{cor}\label{inv3} Under the assumptions of Proposition
\ref{inv2}, there exists $q\in C_0^\infty (\Omega ;{\bf R})$ such that
$M_q:{\bf C}^N\to{\bf C}^N$ is bijective, where \ekv{inv.0.5} { M_qu=\int
q(x)(u|\overrightarrow{f}(x))\overrightarrow{e}(x)dx.  }
\end{cor}
\begin{proof} It suffices to let $q(x)$ be very close to
$\sum_1^N\delta (x-a_j)$ in the weak measure sense.
\end{proof}

We observe that $M$ has the matrix
\ekv{inv.0.7}
{M_{j,k}=\sum_{\nu =1}^N e_j(a_\nu )\overline{f}_k(a_\nu )}
and that $M_q$ has the matrix
$$M_{q,j,k}=
\int q(x)e_j(x)\overline{f}_k(x)dx.
$$

\par We now look for quantitative versions of the preceding results. 
\begin{lemma}
\label{inv4} Let $e_1,...,e_N$ be as in Proposition \ref{inv1} and
also square integrable.
Let $L\subset {\bf C}^N$ be a linear subspace of dimension $M-1$,
for some $1\le M\le N$. Then there exists $x\in \Omega $ such that 
\ekv{inv.3} { \mathrm{dist\,}(\overrightarrow{e}(x), L)^2 \ge
\frac{1}{\mathrm{vol\,}(\Omega )}\mathrm{tr\,}((1-\pi _L){\cal
  E}_\Omega ) ,}
where ${\cal E}_\Omega =((e_j|e_k)_{L^2(\Omega )})_{1\le j,k\le N}$
and $\pi _L$ is the orthogonal projection from ${\bf C}^N$ onto $L$.
\end{lemma}
\begin{proof} 
Let $\nu _1,...,\nu _N$ be an orthonormal basis in ${\bf C}^N$ such that
$L$ is spanned by $\nu _1,...,\nu _{M-1}$ (and equal to $0$ when
$M=1$). Let $(\cdot |\cdot\cdot )$ denote the usual scalar product on
${\bf C}^N$ and let $(\cdot |\cdot\cdot )_\Omega $ be the scalar product on
$L^2(\Omega )$. Write 
$$\nu _\ell=\left(\begin{array}{ccc}\nu _{1,\ell}\\ ..\\ ..\\ \nu_{N,\ell} 
\end{array}\right) .$$
We have 
\begin{eqnarray*}
\mathrm{dist\,}(\overrightarrow{e}(x),L)^2&=&\sum_{\ell=M}^N|(\overrightarrow{e}(x)|\nu _\ell )|^2\\
&=&\sum_{\ell=M}^N |\sum_j e_j(x)\overline{\nu }_{j,\ell}|^2\\
&=&\sum_{\ell=M}^N \sum_{j,k}\overline{\nu }_{j,\ell}
e_j(x)\overline{e}_k(x)\nu _{k,\ell}.
\end{eqnarray*} 
It follows that 
$$
\int_\Omega \mathrm{dist\,}(\overrightarrow{e}(x),L)^2dx=\sum_{\ell =M}^N ({\cal
  E}_\Omega \nu _\ell|\nu _\ell )= \mathrm{tr\,}((1-\pi _L){\cal
  E}_\Omega ). 
$$
It then suffices to estimate the integral from above by
$$\mathrm{vol\,}(\Omega ) \sup_{x\in \Omega }\mathrm{dist\,}
(\overrightarrow{e}(x),L)^2.$$ 
If $\mathrm{dist\,}(\overrightarrow{e}(x),L)^2$ is constant, then any
$x\in \Omega $ will satisfy (\ref{inv.3}), if not,
$$\mathrm{tr\,}((1-\pi _L){\cal E}_\Omega )<\mathrm{vol\,}(\Omega
)\sup_\Omega \mathrm{dist\,}(\overrightarrow{e}(x),L)^2$$ and we can
find an $x\in \Omega $ satisfying (\ref{inv.3}).
\end{proof}

\par If we make the assumption that
 \ekv{inv.1} { e_1,...,e_N \mbox{ is an orthonormal family in }
L^2(\Omega ), }
then ${\cal E}_\Omega =1$ and \no{inv.3} simplifies to 
\ekv{inv.3.5}
{
\max_{x\in \Omega }\mathrm{dist\,}(\overrightarrow{e}(x),L)^2\ge
\frac{N-M+1}{\mathrm{vol\,}(\Omega )}.
}

In the general case, let $0\le\varepsilon _1\le \varepsilon _2\le ...\le
\varepsilon _N$ denote the eigenvalues of ${\cal E}_\Omega $. Then we
have 
\ekv{inv.1.6}
{
\inf_{\mathrm{dim\,}L=M-1}\mathrm{tr\,}((1-\pi _L){\cal E}_\Omega )=
\varepsilon _1+\varepsilon _2+...+\varepsilon _{N-M+1}=:E_M.
}
Indeed, the min-max principle shows that 
$$
\varepsilon _k=\inf_{\mathrm{dim\,}L'=k}\sup_{\nu \in L'\atop \| \nu
  \|=1} ({\cal E}_\Omega \nu |\nu ),
$$
so for a general subspace $L$ of dimension $M-1$, the eigenvalues of
$(1-\pi _L){\cal E}_\Omega (1-\pi _L)$ are $\varepsilon '_1\le ...\le
\varepsilon _{N-M+1}'$, with $\varepsilon '_j\ge \varepsilon _j$. 

Now, we can use the lemma to choose successively $a_1,...,a_N\in \Omega
$ such that
\begin{eqnarray*} \| \overrightarrow{e}(a_1)\Vert^2&\ge& {E_1\over {\rm vol\,}(\Omega
)},\\ {\rm dist\,}(\overrightarrow{e}(a_2),{\bf C} \overrightarrow{e}(a_1))^2&\ge& {E_2\over {\rm
vol\,}(\Omega )},\\ &...&\\ {\rm dist\,}(\overrightarrow{e}(a_M),{\bf C} \overrightarrow{e}(a_1)\oplus
...\oplus{\bf C} \overrightarrow{e}(a_{M-1}))^2 &\ge& {E_M\over {\rm vol\,}(\Omega )},\\
&...&
\end{eqnarray*}
\par Let $\nu _1,\nu _2,...,\nu _N$ be the Gram-Schmidt
orthonormalization of the basis $\overrightarrow{e}(a_1), \overrightarrow{e}(a_2),..., \overrightarrow{e}(a_N)$, so that
\ekv{inv.4} {\overrightarrow{e}(a_M)\equiv c_M \nu _M {\rm mod\,}(\nu
_1,...,\nu_{M-1}), \mbox{ where } | c_M | \ge \left(\frac{E_M}{{\rm
vol\,}(\Omega )}\right)^{1\over 2}.}

\par Consider the $N\times N$ matrix $E=(\overrightarrow{e}(a_1)\, \overrightarrow{e}(a_2)\, ...\,
\overrightarrow{e}(a_N))$ where $\overrightarrow{e}(a_j)$ are viewed
as columns. Expressing these vectors
in the basis $\nu _1,...,\nu _N$ will not change the absolute value of
the determinant and $E$ now becomes an upper triangular matrix with
diagonal entries $c_1,...,c_N$. Hence \ekv{inv.6} {| \det E| = |
c_1\cdot ...\cdot c_N| , } and \no{inv.4} implies that 
\ekv{inv.7} {
|\det E| \ge {(E_1E_2...E_N)^{1/2}\over ({\rm vol\,}(\Omega ))^{N/2}}.  }

We now return to $M$ in \no{inv.0}, (\ref{inv.0.7}) and observe that 
\ekv{inv.8} {
M=E\circ F^*, } where 
\ekv{inv.9} { F=(\overrightarrow{f}(a_1)...\overrightarrow{f}(a_N)).  }
Now, we assume
\ekv{inv.2} { f_j=\overline{e}_j,\ \forall j.  }
Then $F^*=\trans{E}$, so
\ekv{inv.10} { M=E\circ \trans{E}.  }  
We get from \no{inv.7}, \no{inv.10}, that 
\ekv{inv.10.5}
{
|
\det M | \ge {E_1E_2...E_N\over {\rm vol\,}(\Omega )^N}.
}
Under the assumption \no{inv.1}, this simplifies to
\ekv{inv.11} { |
\det M | \ge {N!\over {\rm vol\,}(\Omega )^N}.  }

\par It will also be useful to estimate the singular values $s_1(M)\ge
s_2(M)\ge ...\ge s_N(M)$ of the matrix $M$ (by definition the
decreasing sequence of eigenvalues of the matrix
$(M^*M)^{\frac{1}{2}}$). 
Clearly, 
\ekv{inv.11.1}
{
s_1^N\ge s_1^{k-1}s_k^{N-k+1}\ge \prod _1^Ns_j=|\det M|,\quad 1\le k\le N,
}
and we recall that 
\ekv{inv.11.2}{s_1=\Vert M\Vert .}

\par Combining \no{inv.10.5} and \no{inv.11.1}, we get
\begin{prop}\label{inv5} Under the above assumptions,
\ekv{inv.11.2.5}{
s_1\ge \frac{(E_1...E_N)^{\frac{1}{N}}}{\mathrm{vol\,}(\Omega )},
}
\ekv{inv.11.3}{s_k\ge s_1\left(\prod_1^N\left(\frac{E_j}{s_1\mathrm{vol\,}(\Omega )}\right)
\right)^{\frac{1}{N-k+1}}.}
\end{prop}

\section{Singular values of matrices associated to suitable 
admissible potentials}\label{cl}
\setcounter{equation}{0}
In this section, we let $P,\widetilde{P},p,\widetilde{p}$ be as in the
introduction. (The assumption (\ref{int.6.2}) will not be used here.)
We also choose $\chi _0(x)$, $\epsilon _k$, $\mu _k$, $D=D(h)$,
$L=L(h)$ as in and around (\ref{int.6.3}), (\ref{int.6.4}).
\begin{dref}\label{spe01}
An admissible potential is a potential of the form
\ekv{cl.1}
{
q(x)=\chi _0(x)\sum_{0<\mu _k\le L}\alpha _k\epsilon _k(x),\ \alpha
\in {\bf C}^D.
}
\end{dref}
Here we shall take another step in the construction of an admissible
potential $q$ for which the singular values of $P+\delta h^{N_1}q$ (cf
(\ref{int.6.4.5})) satisfy nice lower bounds. More precisely, we
shall approximate $\delta $-potentials in $H^{-s}$ with admissible
ones and then apply the results of the preceding two sections. Let us
start with the approximation. As in the introduction we let $s>n/2$,
$0<\epsilon <s-n/2$.
\begin{prop}\label{spe02}
Let $a\in\{ x\in {\bf R}^n;\, \chi _0(x)=1\}$. Then $\exists \alpha
\in {\bf C}^D$, $r\in H^{-s}$ such that 
\ekv{cl.7}
{
\delta _a(x)=\chi _0(x)\sum_{\mu _k\le L}\alpha _k\epsilon_k+\chi _0(x)r(x),
}
where
\ekv{cl.9}
{
\Vert \chi _0r\Vert_{H^{-s}}\le C_{s,\epsilon }
L^{-(s-\frac{n}{2}-\epsilon )}h^{-\frac{n}{2}},
}
\ekv{cl.12}
{
(\sum \vert \alpha _k\vert^2)^{\frac{1}{2}}
\le \langle L\rangle^{\frac{n}{2}+\epsilon} (\sum_{\mu _k\le L}\langle
\mu _k\rangle^{-2(\frac{n}{2}+\epsilon ) }
|\alpha _k|^2)^{\frac{1}{2}}
\le CL^{\frac{n}{2}+\epsilon }  h^{-\frac{n}{2}}.
}
\end{prop}
\begin{proof}
Observe first that if $\delta _a=\delta (x-a)$ for
some fixed $a\in {\bf R}^n$, and $s>\frac{n}{2}$ is fixed as in the 
introduction,
\ekv{cl.6.5}
{
\Vert \delta _a\Vert_{H^{-s}}={\cal O}(1)\Vert \langle h\xi 
\rangle^{-s}\Vert_{L^2}={\cal O}_s(1)h^{-\frac{n}{2}}.
}

\par In general, if $u\in H^{-s_1}(\widetilde{\Omega } )$, $s_1>\frac{n}{2}$,  
then Proposition \ref{al2} (where $s$ is arbitrary) shows that 
$$
u=\sum_1^\infty \alpha _k\epsilon_k,\quad \sum \langle \mu _k\rangle^{-2s_1}
|\alpha _k|^2\asymp \Vert u\Vert_{H^{-s_1}}^2.
$$
Thus, if $s>s_1$:
\ekv{cl.6.7}{u=\sum_{\mu _k\le L}\alpha _k\epsilon_k +r,}
where 
\ekv{cl.6.8}
{\Vert r \Vert_{H^{-s}}^2=\sum_{\mu _k>L}\langle \mu
_k\rangle^{-2s}|\alpha _k|^2\le C L^{-2(s-s_1)}\Vert u\Vert_{H^{-s_1}}^2,}
\ekv{cl.11}
{
(\sum_{\mu _k\le L} \vert \alpha _k\vert^2)^{\frac{1}{2}}
\le \langle L\rangle^s(\sum_{\mu _k\le L}\langle \mu _k\rangle^{-2s}
|\alpha _k|^2)^{\frac{1}{2}}\le
CL^s\Vert u\Vert_{H^{-s}}.
}
In particular, when $u=\delta _a$, $a\in K$, we can multiply
(\ref{cl.6.7}) with $\chi _0$ and we get the proposition with
$s_1=\epsilon +n/2$
\end{proof}

\par Let $P_\delta $ be as in \no{hs.3} and assume \no{hs.4}, (\ref{hs.6}).  Let 
${\cal R}(\pi _\alpha )={\bf C}e_1\oplus ...\oplus {\bf C}e_N$ be as
in one of the two cases of Proposition \ref{hs1}. By the mini-max
principle and standard spectral asymptotics (see \cite{DiSj}), we know
that $N={\cal O}(h^{-n})$ and if we want to use the assumption
(\ref{int.6.2}) we even have $N={\cal O}((\max (\alpha ,h))^\kappa
h^{-n})$ by Proposition \ref{gr2}. For the moment we shall only use
that $N$ is bounded by a negative power of $h$. Recall that we have \no{hs.13}, where 
$s>\frac{n}{2}$ is the fixed number appearing in \no{hs.4}. 

\par 
Let $V$ be a fixed neighborhood of the set $K$ in \no{hs.16},
which, as we have seen, can be assumed to be contained in any fixed
given neighborhood of $\overline{\pi _xp^{-1}(\Omega )}$, where 
$\Omega$ is the set in the 
introduction. Let 
$a=(a_1,...,a_N)\in V^N$ and put 
\ekv{spe.01}
{
q_a(x)=\sum_1^N \delta (x-a_j),
}
\ekv{spe.02}
{
M_{q_a;j,k}=\int q_a(x)e_k(x)e_j(x)dx,\ 1\le j,k\le N.
}
Then using \no{hs.13}, \no{al.2} and the fact that $\Vert
q_a\Vert_{H^{-s}}={\cal O}(1)Nh^{-n/2}$, we get for all $\lambda ,\mu \in
{\bf C}^n$,
\begin{eqnarray*}
\langle M_{q_a}\lambda ,\mu \rangle&=&\int q_a(x)(\sum \lambda
_ke_k)(\sum \mu _je_j)dx\\
&=& {\cal O}(1)Nh^{-n}\Vert \lambda \Vert \Vert \mu \Vert
\end{eqnarray*}
and hence
\ekv{spe.03}{s_1(M_{q_a})=\Vert M_{q_a}\Vert_{{\cal L}({\bf C}^N,{\bf
      C}^N)}={\cal O}(1)Nh^{-n}.}

\par We now choose $a$ so that (\ref{inv.11.2.5}), \no{inv.11.3} hold, where we recall
that $s_k$ is the $k$:th singular value of $M_{q_a}$ and $E_j$ is
defined in \no{inv.1.6}, where $0\le\varepsilon _1\le \varepsilon _2\le ...\le
\varepsilon _N$ are the eigenvalues of the Gramian ${\cal E}_V
=((e_j| e_k)_{L^2(V )})_{1\le j,k\le N}$. 

\par From Proposition \ref{hs1} we see that ${{e_j}_\vert}_{V
}$ almost form an orthonormal system in $L^2(V )$: ${\cal E}_V
=1+{\cal O}(h^\infty )$. Hence,
\ekv{spe.04}
{
E_j=N-j+1+{\cal O}(h^\infty ).
}
Then \no{inv.11.2.5} gives the lower bound
\ekv{spe.05}
{
s_1\ge \frac{(1+{\cal O}(h^\infty
  ))(N!)^{\frac{1}{N}}}{\mathrm{vol\,}(V )}
= 
(1+{\cal O}(\frac{\ln N}{N}))\frac{N}{e\,\mathrm{vol\,}(V )}
,
}
where the last identity follows from Stirling's formula.

\par Rewriting \no{inv.11.3} as
$$
s_k\ge s_1^{-\frac{k-1}{N-k+1}}\left( \prod_1^N
  \frac{E_j}{\mathrm{vol\,}(V )}\right)^{\frac{1}{N-k+1}},
$$
and using \no{spe.03}, we get
\ekv{spe.06}
{
s_k\ge \frac{(1+{\cal O}(h^\infty
  ))}{C^{\frac{k-1}{N-k+1}}(\mathrm{vol\,}(V ))^{\frac{N}{N-k+1}}}
\left( \frac{h^n}{N} \right)^{\frac{k-1}{N-k+1}}(N!)^{\frac{1}{N-k+1}}.
}

\par Summing up, we get
\begin{prop}\label{spe1} Let $V$ be a fixed neighborhood of the set
  $K$ in (\ref{hs.16}) (which can be assumed to be contained in any
  fixed given neighborhood of $\pi _x^{-1}(\overline{\Omega })$). 
We can find $a_1,...,a_N\in V $ such that if 
$q_a=\sum_1^N \delta (x-a_j)$ and $M_{q_a;j,k}=\int q_a(x)
e_k(x)e_j(x)dx$, then the singular values $s_1\ge s_2\ge ...\ge s_N$ of $M_{q_a}$,
satisfy \no{spe.03}, \no{spe.05} and \no{spe.06}.
\end{prop}

\par We shall next approximate $q_a$ with an admissible potential. Apply
Proposition \ref{spe02} to each $\delta $-function in $q_a$,
to see that 
\ekv{spe.4}
{
q_a=q+r,\ q=\chi _0(x)\sum_{\mu _k\le L}\alpha _k\epsilon _k,
}
where 
\ekv{spe.4.5}
{
\Vert q\Vert_{H^{-s}}\le Ch^{-\frac{n}{2}}N,
}
\ekv{spe.5}
{
\Vert r\Vert_{H^{-s}}\le C_\epsilon L^{-(s-\frac{n}{2}-\epsilon )}h^{-\frac{n}{2}}N,
}
\ekv{spe.6}
{
(\sum \vert \alpha _k\vert^2)^{\frac{1}{2}}\le
CL^{\frac{n}{2}+\epsilon }h^{-\frac{n}{2}}N.
}
Below, we shall have $N={\cal O}(h^{\kappa -n})$ so if we choose $L$
as in (\ref{int.6.4}), we get 
$$
|\alpha |_{{\bf C}^D}\le C h^{-(\frac{n}{2}+\epsilon )M+\kappa -\frac{3n}{2}}
$$
and $q$ becomes an admissible potential in the sense of
(\ref{int.6.3}), (\ref{int.6.4}).

\par In order to estimate $M_r$, we write
$$
\langle M_r \beta ,\gamma \rangle =\int r(x)(\sum \beta _ke_k)(\sum
\gamma _je_j) dx,
$$
so that
$$
\vert \langle M_r\beta ,\gamma \rangle\vert
\le 
C\Vert r\Vert_{H^{-s}}h^{-\frac{n}{2}}\Vert \sum \beta
_ke_k\Vert_{H^s} \Vert \sum \gamma _je_j\Vert_{H^s}.
$$
Applying \no{hs.13} to the last two factors, we get with a new
constant $C>0$:
$$
\vert \langle M_r\beta ,\gamma \rangle\vert
\le C\Vert r\Vert_{H^{-s}}h^{-\frac{n}{2}}\Vert \beta \Vert \Vert
\gamma \Vert,
$$
so
\ekv{spe.7}
{
\Vert M_r\Vert\le Ch^{-\frac{n}{2}}\Vert r\Vert_{H^{-s}}.
}
Using \no{spe.5}, we get for every $\epsilon >0$
\ekv{spe.8}
{
\Vert M_r\Vert\le C_\epsilon L^{-(s-\frac{n}{2}-\epsilon )}h^{-n}N.
}

\par For the admissible potential $q$ in \no{spe.4}, we thus obtain
from \no{spe.06}, \no{spe.8}:
\ekv{spe.9}
{
s_k (M_q)\ge  \frac{(1+{\cal O}(h^\infty
  ))}{C^{\frac{k-1}{N-k+1}}(\mathrm{vol\,}(V ))^{\frac{N}{N-k+1}}}
\left( \frac{h^n}{N} \right)^{\frac{k-1}{N-k+1}}(N!)^{\frac{1}{N-k+1}}-C_\epsilon L^{-(s-\frac{n}{2}-\epsilon )}h^{-n}N.
}

\par Similarly, from (\ref{spe.03}), (\ref{spe.8}) we get for $L\ge 1$:
\ekv{spe.10}
{
\Vert M_q\Vert \le CNh^{-n}.
}

\par Using Proposition \ref{al2}, we get for all $s_1>n/2$,
\begin{eqnarray*}
\Vert q\Vert_{H^s}&\le &{\cal O}(1) (\sum_{\mu _k\le L}\langle \mu
_k\rangle^{2s}|\alpha _k|^2)^{\frac{1}{2}}\\
&\le & {\cal O}(1) (\sum_{\mu _k\le L} \langle \mu
_k\rangle^{-2s_1}|\alpha _k|^2)^{\frac{1}{2}} L^{s+s_1}\\
&\le &{\cal O}(1) h^{-\frac{n}{2}}NL^{s+s_1},
\end{eqnarray*}
where we used \no{spe.4.5} or rather its proof in the last step. Thus
for every $\epsilon >0$,
\ekv{spe.11}
{
\Vert q\Vert_{H^s}\le {\cal O}(1)NL^{s+\frac{n}{2}+\epsilon }
h^{-\frac{n}{2}},\ \forall \epsilon >0.
}
Summing up, we have obtained
\begin{prop}\label{spe2}
Fix $s>n/2$ and $P_\delta $ as in \no{hs.3}, \no{hs.4}, (\ref{hs.6}) and let $\pi
_\alpha $, $e_1,...,e_N$  be as in one of the two cases in Proposition
\ref{spe02}. Let $V\Subset {\bf R}^n$ be a fixed open 
neighborhood of $K$ in \no{hs.16} and let $\chi _0\in C_0^\infty (V)$
be equal to 1 near $K$. Choose the $h$-dependent parameter
$L$ with $1\ll L\le {\cal O}(h^{-N_0})$ for some fixed $N_0>0$. Then
we can find an admissible potential $q$ as in \no{spe.4} (different
from the one in \no{hs.3}, \no{hs.4}) such that the matrix $M_q$, defined by
$$
M_{q;j,k}=\int q e_ke_j dx,
$$
satisfies \no{spe.9}, \no{spe.10}. Moreover the $H^s$-norm of $q$ 
satisfies \no{spe.11}.
\end{prop}

Notice also that if we choose $\widetilde{R}$ with real coefficients,
then we can choose $q$ real-valued.

\section{Lower bounds on the small singular values for 
suitable perturbations}\label{sv}
\setcounter{equation}{0}

As before, we let 
\ekv{sv.0}{P\sim p+hp_1+... \in S(m),\ p,p_j\in S(m),
}
where $m\ge 1$ is an order function on ${\bf R}^{2n}$. 
We assume that $p-z$ is elliptic for at least one value of $z\in {\bf C}$
and define $\Sigma , \Sigma _\infty $ as in the introduction. Let
$\Omega \Subset {\bf C}$ be open, simply connected with $\Omega
\not\subset \Sigma $, $\overline{\Omega }\cap \Sigma _\infty
=\emptyset $.

\par In this section, we fix a $z\in \Omega $. We will use Proposition
\ref{spe2} iteratively to construct a special admissible perturbation
$P_\delta $ for which we have nice lower bounds on the small singular
values of $P_\delta -z$, that will lead to similar bounds for the ones
of $P_{\delta ,z}$ and to a lower bound on $|\det P_{\delta ,z}|$. 

\par We will need the symmetry assumption (\ref{int.6}). Recall that
$P$ also denotes the $h$-Weyl quantization of the symbol $P$. On the
operator level, (\ref{int.6}) is equivalent to the property
\ekv{sv.2}
{
P^*=\Gamma \circ P\circ\Gamma ,
}
where $\Gamma u=\overline{u}$ denotes the antilinear operator of
complex conjugation.
Notice that the equivalent conditions
\no{int.6}, \no{sv.2} remain unchanged if we add a multiplication
operator to $P$.

\par As in the introduction, we introduce
\ekv{sv.3}
{
V_z(t)=\mathrm{vol\,}(\{ \rho \in {\bf R}^{2n};\, |p(\rho )-z|^2\le
t\}),
}
and assume (for our fixed value of $z$) that 
\ekv{sv.4}
{
V_z(t)={\cal O}(t^\kappa
),\ 0\le t\ll 1,
}
for some $\kappa \in ]0,1]$.
It is easy to see that this assumption is equivalent to
\no{grny.3}. Moreover, from Proposition \ref{gr2} (or directly from 
\cite{HaSj}) it is easy to get,
\begin{prop}\label{sv1}
Assume \no{sv.4} (or  equivalently \no{grny.3}) and recall Remark
\ref{gr0}. For $0<h\ll \alpha \ll 1$, the number $N(\alpha )$ of
eigenvalues of $(P-z)^*(P-z)$ in $[0,\alpha ]$ satisfies
\ekv{sv.5}
{
N(\alpha )={\cal O}(\alpha ^\kappa h^{-n}).
}
\end{prop}
\begin{proof}
If $e\in {\cal D}(P)$ is normalized in $L^2$ and $\Vert (P-z)e\Vert
\le \alpha ^{\frac{1}{2}}$ then $\Vert
(\widetilde{P}-z)^{-1}(P-z)e\Vert \le (C\alpha )^{\frac{1}{2}}$ for
some constant $C>0$. By the minimax principle, it follows that the
number of eigenvalues of $(P-z)^*(P-z)$ in $[0,\alpha ]$ is smaller
than or equal to the number of eigenvalues of $P_z^*P_z$ in
$[0,C\alpha ]$, (where $P_z=(\widetilde{P}-z)^{-1}(P-z)$) and it
suffices to apply Proposition \ref{gr2}.
\end{proof}

\par Let $\epsilon >0$, $s>\frac{n}{2}+\epsilon $ be fixed as in the
introduction and consider
\ekv{sv.5.5}
{
P_0=P+\delta _0q_0,\hbox{ with }0\le \delta _0\ll h,\ \Vert
q_0\Vert _{H^s}\le h^{\frac{n}{2}}.}
 From the mini-max principle, we see that Proposition
\ref{sv1} still applies after replacing $P$ by $P_0$.

\par Choose $\tau_0\in
]0,(Ch)^\frac{1}{2}]$ and let $N={\cal O}(h^{\kappa -n})$ be the number
of singular values of $P_0-z$; $0\le t_1(P_0-z)\le ...\le t_N(P_0-z)<
\tau_0$ in the interval $[0,\tau_0[$. 
As in the introduction we put 
\ekv{sv.14b}
{
N_1=\widetilde{M}+sM+\frac{n}{2},
}
where $M,\widetilde{M}$ are the parameters 
in (\ref{int.6.4}).
Fix $\theta \in ]0,\frac{1}{4}[$ and recall that $N$ is determined by
the property $t_N(P_0-z)<\tau_0\le t_{N+1}(P_0-z)$. Fix $\epsilon
_0>0$.
\begin{prop}\label{sp1} a) If $q$ is an admissible potential as in
  (\ref{int.6.3}), (\ref{int.6.4}), we
  have 
\ekv{sv.14a}
{
\Vert q\Vert_\infty \le Ch^{-n/2}\Vert q\Vert_{H^s}\le \widetilde{C}h^{-N_1}.
}
\par\noindent b) If $N$ is sufficiently large, 
there exists an admissible potential $q$ as in (\ref{int.6.3}),
(\ref{int.6.4}), such that if
$$P_\delta =P_0+\frac{\delta h^{N_1}}{\widetilde{C}}q=:P_0+\delta Q,\
\delta =\frac{\tau_0}{C}h^{N_1+n}$$
(so that $\Vert Q\Vert \le 1$) then 
\ekv{sp.15}
{
t_\nu (P_\delta -z )\ge t_\nu (P_0-z)-\frac{\tau_0h^{N_1+n}}{C}\ge
(1-\frac{h^{N_1+n}}{C})t_\nu (P_0-z),\ \nu >N ,
}
\ekv{sp.16}
{
t_\nu (P_{\delta }-z)\ge \tau_0 h^{N_2},\ [N-\theta N] +1 \le \nu \le N.
}
Here, we put 
\ekv{sp.11}
{
N_2=2(N_1+n)+\epsilon _0,
}
and we let $[a]=\max ({\bf Z}\cap ]-\infty ,a])$ denote the integer
part of the real number $a$.
When $N={\cal O}(1)$, we have the same result provided that we replace
(\ref{sp.16}) by 
\ekv{sp.16a}
{
t_N (P_{\delta })\ge \tau_0 h^{N_2}.}
\end{prop}
\begin{proof}
The part a) follows from Section \ref{al}, the definition of
admissible potentials in the introduction and from the definition of
$N_1$ in (\ref{sv.14b}). (See also (\ref{spe.11}).) We shall therefore concentrate on the proof
of b).

Let $e_1,...,e_N$ be an
orthonormal family of eigenfunctions corresponding to $t_\nu (P_0-z)$, 
so that 
\ekv{sv.6}
{
(P_0-z)^*(P_0-z)e_j=(t_j(P_0-z))^2e_j.
}
Using the symmetry assumption \no{int.6} $\Leftrightarrow$ \no{sv.2}, we see that a
corresponding family of eigenfunctions of $(P-z)(P-z)^*$ is given by 
\ekv{sv.7}
{
\widetilde{f}_j=\Gamma e_j.
}
If the non-vanishing $t_j$ are not all distinct it is not immediately
clear that we can arrange so that $\widetilde{f}_j=f_j$ in \no{gr.16},
but we know that $\widetilde{f}_1,...,\widetilde{f}_N$ and $f_1,...,f_N$ are 
orthonormal families  that span the same space $F_N$. Let $E_N$ be the 
span of $e_1,...,e_N$. We then know that 
\ekv{sv.8}
{
(P_0-z):E_N\to F_N \hbox{ and }(P_0-z)^*:F_N\to E_N
}
have the same singular values $0\le t_1\le t_2\le ...\le t_N$. 

\par Define $R_+:L^2\to {\bf C}^N$, $R_-:{\bf C}^N\to L^2$ by
\ekv{sv.9}
{
R_+u(j)=(u|e_j),\quad R_-u_-=\sum_1^N u_-(j)\widetilde{f}_j.
}
Then 
\ekv{sv.9.5}
{
{\cal P}=\left(\begin{array}{ccc}P_0-z &R_-\\ R_+ &0 \end{array}\right)
:{\cal D}(P_0)\times {\bf C}^N \to L^2\times {\bf C}^N
}
has a bounded inverse 
$$
{\cal E}=\left(\begin{array}{ccc}E &E_+\\ E_-
    &E_{-+} \end{array}\right) .
$$
Since we do not necessarily have \no{gr.16} we cannot say that
$E_{-+}=
\mathrm{diag\,}(t_j)$ but we know that the singular values of $E_{-+}$
are given by
$
t_j(E_{-+})=t_j(P_0-z),\ 1\le j\le N,
$ 
or equivalently by $s_j(E_{-+})=t_{N+1-j}(P_0-z)$, for $1\le j\le N$.
   
\par We will apply Section \ref{gr}, and recall that $N$ is assumed to
be sufficiently large and that $\theta $ has been fixed in $]0,1/4[$.
(The case of bounded $N$ will be treated later.) 
Let $N_2$ be given in (\ref{sp.11}). Since $z$ is fixed it will also be
notationally convenient to assume that $z=0$.

\medskip\par\noindent 
\emph{Case 1.} $s_j(E_{-+})\ge \tau_0h^{N_2}$, 
for $1\le j \le N-[(1-\theta )N]$. Then we get the proposition with
$q =0$, $P_\delta =P_0$.

\medskip
\par\noindent \emph{Case 2.}
\ekv{sp.2}
{
s_j(E_{-+})<\tau_0h^{N_2} \hbox{ for some }j\hbox{ such that }1\le
j \le N-[(1-\theta )N].
}

Recall that for the special 
admissible potential $q$ in \no{spe.4}, we have
\no{spe.9}. For $k\le N/2$, we have $N-k+1>N/2$, so 
$$
\frac{k-1}{N-k+1}\le 1,
$$
and \no{spe.9} gives
$$
s_k(M_q)\ge \frac{1+{\cal O}(h^\infty )}{C}
\frac{h^n}{N}(N!)^\frac{1}{N}-C_\epsilon
L^{-(s-\frac{n}{2}-\epsilon )}\frac{N}{h^n}.
$$
By Stirling's formula, we have $(N!)^\frac{1}{N}\ge N/\mathrm{Const}$,
so for $1\le k\le N/2$, we obtain with a new constant $C>0$:
$$
s_k(M_q)\ge \frac{h^n}{C}-C_\epsilon
L^{-(s-\frac{n}{2}-\epsilon )}\frac{N}{h^n}.
$$
Here, we recall from Proposition \ref{sv1} (which also applies to
$P_0$) that $N={\cal O}(h^{\kappa
  -n})$
and choose $L$ so
that 
$$
L^{-(s-\frac{n}{2}-\epsilon )}h^{\kappa -2n}\ll h^n,
$$
i.e. so that (in agreement with \no{int.6.4}) 
\ekv{sv.10}
{
L\gg h^\frac{\kappa -3n}{s-\frac{n}{2}-\epsilon }.
}
We then get 
\ekv{sv.11}
{
s_k(M_q)\ge \frac{h^n}{C},\ 1\le k\le \frac{N}{2},
}
for a new constant $C>0$.

\par From (\ref{spe.10}) and the fact that $N={\cal O}(h^{\kappa -n})$
we get
\ekv{sv.12}
{
s_1(M_q)\le \Vert M_q\Vert\le CNh^{-n}\le \widetilde{C}h^{\kappa -2n}.
}

\par In addition to the lower bound \no{sv.10} we assume as in 
\no{int.6.4} (in all
cases) that 
\ekv{sv.13}
{
L\le C h^{-M},\mbox{ for some } M\ge \frac{3n-\kappa }{s-\frac{n}{2}-\epsilon }.
}
As we saw after (\ref{spe.6}), $q$ is indeed an admissible potential
as in (\ref{int.6.3}), (\ref{int.6.4}),
so that by (\ref{sv.14a})
\ekv{sv.14}
{
\Vert q\Vert_\infty \le Ch^{-\frac{n}{2}}\Vert q\Vert_{H^s}
\le \widetilde{C}h^{-N_1}.}

Put 
\ekv{sp.3}{P_\delta =P_0+\frac{\delta h^{N_1}}{\widetilde{C}}q=P_0+\delta Q,\
Q=\frac{h^{N_1}}{\widetilde{C}}q,\ \Vert Q\Vert \le 1.}
Then, if $\delta \le \tau_0/2$, we can replace $P_0$ by $P_\delta $
in \no{sv.9.5} and we still have a well-posed problem with inverse as in
\no{gr.6.5}--\no{gr.7.3}, satisfying \no{s.14}--\no{s.16} 
with $Q_\omega =Q$ as
above. Here $E_-^0QE_+^0=h^{N_1}M_q/\widetilde{C}$ 
so according to \no{sv.11},
we have with a new constant $C$
\ekv{sp.4}
{
s_k(\delta E_-^0QE_+^0)\ge \frac{\delta h^{N_1+n}}{C},\ 1\le
k\le \frac{N}{2}.
}

\par Playing with the general estimate \no{s.9}, we get
$$s_\nu (A+B)\ge s_{\nu +k-1}(A)-s_k(B)$$ and for a sum of three operators
$$s_\nu (A+B+C)\ge s_{\nu +k+\ell -2}(A)-s_k(B)-s_\ell (C).$$ We apply this to 
$E_{-+}^\delta $ in \no{s.15} and get 
\ekv{sp.6}
{
s_\nu (E_{-+}^\delta )\ge s_{\nu +k-1}(\delta
E_-^0QE_+^0)-s_k(E_{-+}^0)-2\frac{\delta ^2}{\tau_0}.
} 
Here we use \no{sp.2} with $j=k=N-[(1-\theta )N]$ as well as 
\no{sp.4}, to
get for $\nu  \le N-[(1-\theta )N]$ 
\ekv{sp.7}
{
s_\nu (E_{-+}^\delta )\ge 
\frac{\delta h^{N_1+n}}{C}
-\tau
_0h^{N_2}-2\frac{\delta ^2}{\tau_0}.
}
Recall that $\theta <\frac{1}{4}$.

\par Choose 
\ekv{sp.9}
{
\delta =\frac{1}{C}\tau_0h^{N_1+n},
}
where (the new constant) $C>0$ is sufficiently large.

\par 
Then, with a new constant $C>0$,
we get (for $h>0$ small enough)
\ekv{sp.10}
{
s_\nu (E_{-+}^\delta )\ge \frac{\delta }{C}h^{N_1+n},\ 1\le \nu \le N-
[(1- \theta) N],
}
implying
\ekv{sp.12}
{
s_\nu (E_{-+}^\delta )\ge 8\tau_0h^{N_2},\ 1\le \nu  \le N-
[(1- \theta) N].} For the corresponding
operator $P_\delta $, we have for $\nu >N$:
$$
t_\nu (P_\delta )\ge t_\nu (P_0)-\delta =t_\nu (P_0)-\frac{\tau _0h^{N_1+n}}{C}.
$$ Since $t_\nu (P)\ge \tau_0$ in this case, we get (\ref{sp.15}).

\par From \no{sp.12} and \no{s.16}, we
get (\ref{sp.16}).

\par When $N={\cal O}(1)$, we still get (\ref{sp.7}) with $\nu =1$ and
this leads to (\ref{sp.16a}).

\end{proof}

\par The construction can now be iterated. assume that $N\gg 1$ and 
replace $(P_0,N,\tau_0)$
by $(P_\delta ,[(1-\theta )N], \tau_0h^{N_2})=:
(P^{(1)},N^{(1)},\tau_0^{(1)})$ and keep on, using the same values
for the exponents $N_1,N_2$. Then we get a sequence $(P^{(k)},N^{(k)},\tau
_0^{(k)})$, $k=0,1,...,k(N)$, where the last value $k(N)$ is
determined by the fact that $N^{(k(N))}$ is of the order of magnitude of
a large constant. Moreover,
\ekv{sp.17}
{
t_\nu (P^{(k)})\ge \tau_0^{(k)},\ N^{(k)}<\nu \le N^{(k-1)},
} 
\ekv{sp.18}
{
t_\nu (P^{(k+1)})\ge t_\nu (P^{(k)})-
\frac{\tau_0^{(k)}h^{N_1+\nu }}{C},\ \nu >N^{(k)},
}
\ekv{sp.18.1}
{
\tau_0^{(k+1)}=\tau_0^{(k)}h^{N_2},
}
\ekv{sp.18.2}
{
N^{(k+1)}=[(1-\theta )N^{(k)}],
}
$$
P^{(0)}=P,\ N^{(0)}=N,\ \tau_0^{(0)}=\tau_0.
$$
Here,
\begin{eqnarray*}
&P^{(k+1)}=P^{(k)}+\delta ^{(k+1)}Q^{(k+1)}=P^{(k)}+\frac{\delta
  ^{(k+1)}h^{N_1}}{\widetilde{C}}q^{(k+1)},&\\ &\Vert Q^{(k+1)}\Vert\le 1,\
\delta ^{(k+1)}=\frac{1}{C}\tau_0^{(k)}h^{N_1+n}.&
\end{eqnarray*}
Notice that $N^{(k)}$ decays exponentially fast with $k$:
\ekv{sp.18.5}
{
N^{(k)}\le (1-\theta )^kN,
}
so we get the condition on $k$ that $(1-\theta )^kN\ge C\ll 1$ which
gives,
\ekv{sp.19}{k\le \frac{\ln \frac{N}{C}}{\ln \frac{1}{1-\theta }}.}
We also have 
\ekv{sp.20}
{
\tau_0^{(k)}=\tau_0\left( h^{N_2} \right)^k .
}

\par For $\nu >N$, we iterate \no{sp.18}, to get
\eekv{sp.23}
{t_\nu (P^{(k)})&\ge& t_\nu (P)- \tau_0 \frac{h^{N_1+n}}{C}\left( 1+
h^{N_2}+h^{2N_2}+...\right)}
{&\ge& t_\nu (P)-\tau_0 {\cal O}(\frac{h^{N_1+n}}{C}).}

\par For $1\ll \nu \le N$, let $\ell=\ell (N)$ be the unique value for
which $N^{(\ell )}<\nu \le N^{(\ell -1)}$, so that 
\ekv{sp.24}
{
t_\nu (P^{(\ell )})\ge \tau_0^{(\ell )},
}
by \no{sp.17}. If $k>\ell $, we get 
\ekv{sp.24.5}
{
t_\nu (P^{(k)})\ge t_\nu (P^{(\ell )})-
\tau_0^{(\ell )}{\cal O}(\frac{h^{N_1+n}}{C}) .
}

\par The iteration above works until we reach a value $k=k_0={\cal O}
(\frac{\ln \frac{N}{C}}{\ln \frac{1}{1-\theta }})$ for which
$N^{(k_0)}={\cal O}(1)$. After that, we continue the iteration further 
by decreasing
$N^{(k)}$ by one unit at each step. 

\medskip

\par Summing up the discussion so far, we have obtained
\begin{prop}\label{sp2}
Let $(P,z)$ satisfy the assumptions as in the beginning of this
section and choose $P_0$ as in (\ref{sv.5.5}).
 Let
$s>\frac{n}{2}$, $0<\epsilon <s-\frac{n}{2}$, $M\ge \frac{3n-\kappa
}{s-\frac{n}{2}-\epsilon }$, $N_1=\widetilde{M}+sM+\frac{n}{2}$, 
$N_2=2(N_1+n)+\epsilon _0$, where $\epsilon _0>0$. Let $L$ be an
$h$-dependent parameter satisfying
\ekv{sp.24.8}
{
h^{\frac{\kappa -3n}{s-\frac{n}{2}-\epsilon }}\ll L\le C h^{-M}. 
}
Let $0<\tau_0\le \sqrt{h}$ and let $N^{(0)}={\cal O}(h^{\kappa -n})$ 
be the number of singular
values of $P_0-z$ in $[0,\tau_0[$. Let $0<\theta <\frac{1}{4}$ and let
$N(\theta )\gg 1$ be sufficiently large. Define $N^{(k)}$, $1\le k\le
k_1$ iteratively in the following way. As long as $N^{(k)}\ge
N(\theta )$, we put $N^{(k+1)}=[(1-\theta )N^{(k)}]$. Let $k_0\ge 0$
be the last $k$ value we get in this way. For $k>k_0$ put 
$N^{(k+1)}=N^{(k)}-1$, until
we reach the value $k_1$ for which $N^{(k_1)}=1$.

\par Put $\tau_0^{(k)}=\tau_0h^{kN_2}$, $1\le k\le k_1+1$. Then there 
exists an admissible potential $q=q_h(x)$ as in (\ref{int.6.3}), (\ref{int.6.4}), satisfying \no{spe.6}, \no{spe.11}, so that,
$$
\Vert q\Vert_{H^s}\le {\cal O}(1)h^{-N_1+\frac{n}{2}}, \
\Vert q\Vert_{L^\infty }\le {\cal O}(1)h^{-N_1}, 
$$
such that if $P_{\delta}=P_0+\frac{1}{C}\tau
_0h^{2N_1+n}q=P_0+\delta Q$, $\delta =\frac{1}{C}h^{N_1+n}\tau_0$, $Q=h^{N_1}q$,
we have the following estimates on the singular values of 
$P_{\delta }-z$:
\begin{itemize}
\item If $\nu >N^{(0)}$, we have 
$t_\nu (P_{\delta }-z)\ge (1-\frac{h^{N_1+n}}{C})t_\nu (P_0-z)$.
\item If $N^{(k)}<\nu \le N^{(k-1)},$ $1\le k\le k_1$, then $
t_\nu (P_\delta -z)\ge (1-{\cal O}(h^{N_1+n}))\tau_0^{(k)}$.

\item Finally, for $\nu =N^{(k_1)}=1$, we have  $
t_1(P_\delta -z)\ge (1-{\cal O}(h^{N_1+n}))\tau_0^{(k_1+1)}$.
\end{itemize}
\end{prop}

\par We shall now obtain the corresponding estimates for the singular 
values of $P_{\delta ,z}=(\widetilde{P}_\delta -z)^{-1}(P_\delta -z)$.
Let $e_1,...,e_N$ be an orthonormal family corresponding to
the singular values $t_j(P_\delta )$ in $[0,\sqrt{h}[$, put
$\widetilde{f}_j=\overline{e}_j$ and let 
$$(P_\delta -z) u+R_-u_-=v,\ R_+u=v_+$$
be the corresponding Grushin problem so that the solution
operators fulfil
\ekv{sp.25}
{\Vert E\Vert\le \frac{1}{\sqrt{h}},\ \Vert E_\pm\Vert \le 1,\quad
t_j(E_{-+})=t_j(P_\delta )\le \sqrt{h},\ 1\le j\le N.
} 
Still with $z=0$ we put $\widetilde{R}_-=\widetilde{P}_\delta^{-1}
R_-$. Then the problem
$$
P_{\delta ,z}u+\widetilde{R}_-u_-=v,\ R_+u=v_+,
$$
is wellposed with the solution
$$
u=\widetilde{E}v+\widetilde{E}_+v_+,\    u_-=\widetilde{E}_-v+
\widetilde{E}_{-+}v_+,
$$
where
\begin{eqnarray*}
\widetilde{E}=E\widetilde{P}_\delta &,&\ \widetilde{E}_+=E_+\\
\widetilde{E}_-=E_-\widetilde{P}_\delta &,&\ \widetilde{E}_{-+}=E_{-+}.
\end{eqnarray*}
Adapting the estimate \no{s.12} to our situation, we get
\ekv{sp.26}
{
t_k(P_{\delta ,z})\ge 
\frac{t_k(P_\delta )}{\Vert E \widetilde{P}_\delta 
\Vert t_k(P_\delta )+
\Vert E_+\Vert \Vert E_-\widetilde{P}_\delta \Vert},\ 1\le k\le N,
}
where we also recall that $t_k(P_\delta )\le \sqrt{h}$.

\par Write 
\begin{eqnarray*}
E\widetilde{P}_\delta &=&EP_\delta +E (\widetilde{P}-P)\\
E_-\widetilde{P}_\delta &=&E_-P_\delta +E_- (\widetilde{P}-P)
\end{eqnarray*}
 and use that
\begin{eqnarray*}
EP_\delta =1-E_+R_+={\cal O}(1) \hbox{ in }{\cal L}(L^2,L^2)\\
E_-P_\delta =-E_{-+}R_+={\cal O}(\sqrt{h}) \hbox{ in }{\cal L}(L^2,\ell ^2)
\end{eqnarray*}
together with \no{sp.25} and the fact that $\Vert \widetilde{P}-P\Vert
={\cal O}(1)$. It follows that 
$$
\Vert E\widetilde{P}_\delta \Vert= {\cal O}(\frac{1}{\sqrt{h}}),\ 
\Vert E_-\widetilde{P}_\delta \Vert= {\cal O}(1).
$$
Using this in \no{sp.26}, we get 
\ekv{sp.27}
{
t_k(P_{\delta ,z})\ge \frac{t_k(P_\delta )}{C \frac{t_k(P_\delta )}{\sqrt{h}}+C}
\ge \frac{t_k(P_\delta )}{2C},
}
where used that $t_k(P_\delta )\le \sqrt{h}$ when $1\le k\le N^{(0)}$. Now
the choice of $N_2$ gives us some margin and we can even get rid of
the effect of $2C$ and get for $\tau_0\in ]0,\sqrt{h}]$: 
\begin{prop}\label{sp3}
Proposition \ref{sp2} remains valid if we replace $P_\delta -z$ there
with $P_{\delta ,z}$.
\end{prop}

Consider the operator $P_{\delta ,z}$ in Proposition \ref{sp3}, let 
$\tau_0\in ]0,\sqrt{h}]$ and
choose a corresponding associated Grushin problem 
$$
{\cal P}_\delta =\left(\begin{array}{ccc}P_{\delta ,z} &R_{-,\delta }
\\ R_{+,\delta } &0 \end{array}\right)
$$
as in \no{gr.1}--\no{gr.3} so that \no{grny.1} holds and moreover
for the corresponding inverse 
$\left(\begin{array}{ccc}E &E_+\\E_- &E_{-+} \end{array}\right)$
$$
t_\nu (E_{-+})=t_\nu (P_{\delta ,z}),\ 1\le \nu \le N^{(0)}.
$$
We have 
\ekv{sp.28}
{
|\det E_{-+}| =\prod _1^{N^{(0)}}t_\nu (E_{-+}),
}
and we shall estimate this quantity from below. In the terms of
Proposition \ref{sp2} we have for $1\le k\le k_0$:
$$
N^{(k-1)}-N^{(k)}=N^{(k-1)}-[(1-\theta )N^{(k-1)}]\le \theta N^{(k-1)}+1\le
1+\theta (1-\theta )^{k-1}N^{(0)},
$$
so according to Proposition \ref{sp3} we know that 
$$
\prod_{1+N^{(k)}}^{N^{(k-1)}}t_\nu (E_{-+})\ge ((1-{\cal
  O}(h^{N_1+n)}))\tau_0h^{kN_2})^{1+\theta (1-\theta )^kN^{(0)}}.
$$
For the bounded number of $k$ with $k_0<k\le k_1$, we have 
$N^{(k-1)}-N^{(k)}=1$ and $t_{N^{(k-1)}}(E_{-+})\ge (1-{\cal
  O}(h^{N_1+n}))\tau_0h^{kN_2}$.
Hence from \no{sp.28}:
\begin{eqnarray*}
\ln |\det E_{-+}|&\ge & -\sum_{k=1}^{k_0}({\cal O}(h^{N_1+n})+\ln
\frac{1}{\tau_0}+kN_2\ln \frac{1}{h})(1+\theta (1-\theta )^{k-1}
N^{(0)})\\
&& -\sum_{k_0+1}^{k_1+1}({\cal O}(h^{N_1+n})+\ln \frac{1}{\tau
  _0}+kN_2 \ln \frac{1}{h}).
\end{eqnarray*}
Recall also that $k_1={\cal O}(\ln N^{(0)})={\cal O}(1)\ln
\frac{1}{h}$, and that $N^{(0)}={\cal O}(h^{\kappa -n})$ (by \no{sv.5} with $\alpha ={\cal
  O}(h)$, valid for $P_0$). We get 
\eekv{sp.29}
{\ln |\det E_{-+}|&\ge & -C(\ln \frac{1}{\tau_0}+(\ln \frac{1}{h})^2)
\sum_{k=0}^{k_1}(1+\theta (1-\theta )^kN^{(0)})}
{&\ge & -\widetilde{C}(\ln \frac{1}{\tau_0}+(\ln \frac{1}{h})^2)
(h^{\kappa -n}+\ln \frac{1}{h}).
} 

\par Combining this estimate with \no{grny.1} and \no{grny.9} for
$\alpha =h$, we get
when $\tau_0 =\sqrt{h}$: 
\begin{prop}\label{sp4}
For the special admissible perturbation $P_\delta $ in the
propositions \ref{sp2}, \ref{sp3}, we have 
\eekv{sp.30}
{
&&\ln |\det P_{\delta ,z}|\ge }
{&&\frac{1}{(2\pi h)^n}
\left( \iint \ln |p_z|dxd\xi -{\cal O}\left(h^{N_1+n-\frac{1}{2}}
+(h^{\kappa }+h^n\ln \frac{1}{h})(\ln \frac{1}{\tau_0}+
(\ln \frac{1}{h})^2)\right)
\right) .
}
\end{prop}

\par We also have the upper bound 
$$
|\det E_{-+}| \le \Vert E_{-+}\Vert ^{N^{(0)}}\le 
\exp (CN^{(0)}),
$$
which together with \no{grny.1}, \no{grny.9} gives
\ekv{sp.31}
{
\ln |\det P_{\delta ,z}|\le \frac{1}{(2\pi h)^n}
\left( \iint \ln |p_z|dxd\xi +{\cal O}\left( h^{N_1+n-\frac{1}{2}}
+h^{\kappa }\ln \frac{1}{h}\right)
\right) .
}
Notice that this bound is more general, it only depends on the
fact that the perturbation of $P$ is of the form $\delta Q$ with $\delta
=\tau_0h^{N_1+n}/C$ and with $\Vert Q\Vert ={\cal O}(1)$.

\par When $\tau_0\le \sqrt{h}$
 we keep the same Grushin problem as before and notice that the
 singular values of $E_{-+}$ that are $\le \tau_0$, obey the estimates
 in Proposition \ref{sp2}. Their contribution to $\ln |\det E_{-+}|$
 can still be estimated from below as in \no{sp.29}. The contribution
 from the singular values of $E_{-+}$ that are $>\tau_0$ to $\ln |\det
 E_{-+}|$ can be estimated from below by $-{\cal O}(h^{\kappa -n}\ln
 (1/\tau_0 ))$ and hence \no{sp.29} remains valid in this case. {\it We
 conclude that Proposition \ref{sp4} remains valid when $0<\tau_0\le
 \sqrt{h}$. The same holds for the upper bound \no{sp.31}.}

\section{Estimating the probability that $\det E_{-+}^\delta $ is
small}
\label{pr} \setcounter{equation}{0}

In this section we keep the assumptions on $(P,z)$ of the beginning of
Section \ref{sv} and choose $P_0$ as in (\ref{sv.5.5}). We consider
general $P_\delta $ of the form 
\ekv{pr.1}
{P_\delta =P_0+\delta Q,\ \delta Q=\delta h^{N_1}q(x),\ \delta =\frac{1}{C}h^{N_1+n}\tau_0,
}
where $q$ is an admissible potential as in (\ref{int.6.3}), (\ref{int.6.4}).
Notice that  
$D:=\# \{k;\, \mu _k\le L\}$ satisfies:
\ekv{pr.3}{D \le {\cal O}(L^n h^{-n})\le {\cal O}(h^{-N_3}),\ 
N_3:=n(M+1).
}
With $R$ as in (\ref{int.6.3}), we allow $\alpha $ to vary in the ball
\ekv{pr.4}
{
| \alpha |_{{\bf C}^D}\le 2R={\cal O}(h^{-\widetilde{M}}).
}
(Our probability measure will be supported in $B_{{\bf C}^D}(0,R)$ but
we will need to work in a larger ball.)

\par We consider the holomorphic function 
\ekv{pr.5}
{
F(\alpha )=(\det P_{\delta ,z})\exp (-\frac{1}{(2\pi h)^n}\iint \ln
|p_z| dxd\xi ).
}
 Then by \no{sp.31}, we have
\ekv{pr.6}
{
\ln |F(\alpha )|\le \epsilon _0(h)h^{-n}, \ |\alpha |<2R,
}
and for one particular value $\alpha =\alpha ^0$ with $|\alpha ^0|\le 
\frac{1}{2}R$, corresponding to the special potential in Proposition \ref{sp2}:
\ekv{pr.7}
{
\ln |F(\alpha ^0 )|\ge -\epsilon _0(h)h^{-n},
}
where we put
\ekv{pr.8}
{
\epsilon _0(h)=C\left( h^{N_1+n-\frac{1}{2}}+(h^{\kappa }+h^n\ln \frac{1}{h})
(\ln \frac{1}{\tau_0}+(\ln \frac{1}{h})^2)
\right) .
}
Here $N_1\ge 1/2$ by (\ref{int.6.4.3}) so we can drop the first term in 
(\ref{pr.8}).

\par Let $\alpha ^1\in {\bf C}^D$ with $|\alpha ^1|=R$ and 
consider the holomorphic function of one complex variable
\ekv{pr.9}
{
f(w)=F(\alpha ^0+w\alpha ^1).
}
We will
mainly consider this function for $w$ in the disc 
determined by the condition $|\alpha ^0+w\alpha ^1|<R$:
\ekv{pr.10}
{
D_{\alpha ^0,\alpha ^1}:\left | w+\left( \frac{\alpha ^0}{R} |
\frac{\alpha ^1}{R}\right) \right|^2<1-\left| \frac{\alpha
^0}{R}\right|^2+\left|\left(\frac{\alpha^0}{R}|\frac{\alpha^1}{R}\right)\right|
^2=:r_0^2,}
whose radius is between $\frac{\sqrt{3}}{2}$ and $1$. 

\par From \no{pr.6}, \no{pr.7} we get 
\ekv{pr.11}
{
\ln |f(0)|\ge -\epsilon _0(h) h^{-n},\ 
\ln |f(w)|\le \epsilon _0(h)h^{-n}.
} 
By \no{pr.6}, we may
assume that the last estimate holds in a larger disc, say 
$D(-(\frac{\alpha ^0}{R}|\frac{\alpha ^1}{R}),2r_0)$. Let
$w_1,...,w_M$ be the zeros of $f$ in  
$D(-(\frac{\alpha ^0}{R}|\frac{\alpha ^1}{R}),3r_0/2)$. Then it is
standard to get the factorization 
\ekv{pr.12}
{
f(w)=e^{g(w)}\prod_1^M (w-w_j),\  w\in D(-(\frac{\alpha ^0}{R}|\frac{\alpha ^1}{R}),4r_0/3),
}
together with the bounds
\ekv{pr.13}{|\Re g(w)|\le {\cal O}(\epsilon _0(h)h^{-n}),\ 
M={\cal O}(\epsilon _0(h)h^{-n}).}
See for instance Section 5 in \cite{Sj} where further references are also given.

\par For $0<\epsilon \ll 1$, put 
\ekv{pr.14}
{
\Omega (\epsilon )=\{ r\in [0,r_0[;\, \exists w\in D_{\alpha ^0,\alpha
  ^1}
\hbox{ such that }|w|=r\hbox{ and }|f(w)|<\epsilon \} .
}
If $r\in \Omega (\epsilon )$ and $w$ is a corresponding point in
$D_{\alpha ^0,\alpha ^1}$, we have with $r_j=|w_j|$,
\ekv{pr.14.5}
{\prod_1^M |r-r_j| \le \prod _1^M|w-w_j|\le \epsilon \exp ({\cal
  O}(\epsilon _0(h)h^{-n})).}

Then at least one of the factors $|r-r_j|$ is bounded by 
$ (\epsilon e^{{\cal O}(\epsilon _0(h)h^{-n})})^{1/M}  $. 
In particular, the Lebesgue measure $\lambda (\Omega
(\epsilon ))$ of $\Omega (\epsilon )$ is bounded by 
$2M(\epsilon e^{{\cal O}(\epsilon _0(h)h^{-n})})^{1/M}$. 
Noticing that the last bound increases with $M$ when the last member
of (\ref{pr.14.5}) is $\le 1$,
we get
\begin{prop}\label{pr1} Let $\alpha ^1\in {\bf C}^D$ with 
$|\alpha ^1|=R$ and assume that $\epsilon >0$ is small enough so that the last member of \no{pr.14.5} is
  $\le 1$.
Then 
\eekv{pr.15}
{
\lambda (\{ r\in [0,r_0];\ |\alpha ^0+r\alpha ^1|<R,\ |F(\alpha
^0+r\alpha ^1)|<\epsilon \}) 
\le}
{\frac{\epsilon _0(h)}{h^n}\exp ({\cal O}(1)+\frac{h^n}{{\cal O}(1)
  \epsilon _0(h)}\ln \epsilon ).
}
Here and in the following, the symbol ${\cal O}(1)$ in a denominator
indicates a bounded positive quantity.
\end{prop}

\par Typically, we can choose $\epsilon =\exp -\frac{\epsilon
  _0(h)}{h^{n+\alpha }}$ for some small $\alpha >0$ and then the upper
bound in \no{pr.15} becomes 
$$
\frac{\epsilon _0(h)}{h^n}\exp ({\cal O}(1)-\frac{1}{{\cal O}(1)h^{\alpha }}).
$$

\par
Now we equip $B_{{\bf C}^D}(0,R)$ with a probability measure of the
form
\ekv{pr.16}
{
P(d\alpha )=C(h)e^{\Phi (\alpha )}L(d\alpha ),
} 
where $L(d\alpha )$ is the Lebesgue measure, $\Phi $ is a $C^1$
function which depends on $h$ and satisfies
\ekv{pr.17}
{
\vert \nabla \Phi \vert ={\cal O}(h^{-N_4}),
}
and $C(h)$ is the appropriate normalization constant.

\par Writing $\alpha =\alpha ^0+Rr\alpha ^1$, $0\le r<r_0(\alpha ^1)$,
$\alpha ^1\in S^{2D-1}$, $\frac{\sqrt{3}}{2}\le r_0\le 1$, we get 
\ekv{pr.18}
{
P(d\alpha )=\widetilde{C}(h)e^{\phi (r)}r^{2D-1}dr S(d\alpha ^1),
}
where $\phi (r)=\phi _{\alpha ^0,\alpha ^1}(r)=\Phi (\alpha
^0+rR\alpha ^1)$ so that $\phi '(r)={\cal O}(h^{-N_5})$,
$N_5=N_4+\widetilde{M}$. 
Here
$S(d\alpha ^1)$ denotes the Lebesgue measure on $S^{2D-1}$.

\par For a fixed $\alpha ^1$, we consider the normalized measure 
\ekv{pr.19}{
\mu (dr)=\widehat{C }(h)e^{\phi (r)}r^{2D-1}dr
}
on $[0,r_0(\alpha ^1)]$
and we want to show an estimate similar to \no{pr.15} for $\mu $
instead of $\lambda $. Write 
$e^{\phi (r)}r^{2D-1}=\exp (\phi (r)+(2D-1)\ln r)$ and consider
the derivative of the exponent,
$$
\phi '(r)+\frac{2D-1}{r}.
$$
This derivative is $\ge 0$ for $r\le
2\widetilde{r}_0$, where
$\widetilde{r}_0=C^{-1}\min (1,Dh^{N_5})$ for some large constant
$C$, and we may assume that
$2\widetilde{r}_0\le r_0$. Introduce the measure
$\widetilde{\mu }\ge \mu $ by
\ekv{pr.20}
{
\widetilde{\mu }(dr)=\widehat{C}(h)e^{\phi (r_{\rm max} )}r_{\rm max}^{2D-1}dr,\
r_{\rm max} :=\max (r,\widetilde{r}_0).
}
Since $\widetilde{\mu }([0,\widetilde{r}_0])\le \mu ([\widetilde{r}_0,2\widetilde{r}_0])$, we get
\ekv{pr.21}
{
\widetilde{\mu }([0,r(\alpha ^1)])\le {\cal
  O}(1).
}
We can write 
\ekv{pr.22}
{
\widetilde{\mu }(dr)=\widehat{C}(h)e^{\psi (r)}dr,
}
where 
\eekv{pr.23}
{&
\psi '(r)={\cal O}(\max (D,h^{-N_5}))={\cal O}(h^{-N_6}),&
}
{&N_6=\max (N_3,N_5).&}
Cf (\ref{pr.3}).

\par  We now decompose $[0,r_0(\alpha ^1)]$ into $\asymp h^{-N_6}$
intervals of length $\asymp h^{N_6}$. If $I$ is such an interval, we see that 
\ekv{pr.24}
{
\frac{\lambda (dr)}{C\lambda (I)}\le \frac{\widetilde{\mu
  }(dr)}{\widetilde{\mu }(I)}\le C\frac{\lambda (dr)}{\lambda
  (I)}\hbox{ on }I.
}

From \no{pr.15}, \no{pr.24} we get when the right hand side of 
\no{pr.14.5} is $\le 1$,
\begin{eqnarray*}
\widetilde{\mu }(\{ r\in I;\, |F(\alpha ^0+rR\alpha ^1)|<\epsilon
\})/\widetilde{\mu }(I)&\le& \frac{{\cal O}(1)}{\lambda (I)}
\frac{\epsilon _0(h)}{h^n}\exp (\frac{h^n}{{\cal O}(1)\epsilon _0(h)}\ln
\epsilon )\\
&=& {\cal O}(1) h^{-N_6}
\frac{\epsilon _0(h)}{h^n}\exp (\frac{h^n}{{\cal O}(1)\epsilon _0(h)}\ln
\epsilon ).
\end{eqnarray*}
Multiplying with $\widetilde{\mu }(I)$ and summing the estimates over $I$ we get 
\ekv{pr.25}
{
\widetilde{\mu }(\{ r\in [0,r(\alpha ^1)];\, |F(\alpha ^0+rR\alpha
^1)|<\epsilon \})\le {\cal O}(1)h^{-N_6}\frac{\epsilon _0(h)}{h^n}
\exp (\frac{h^n}{{\cal O}(1)\epsilon _0(h)}\ln \epsilon )
.
}
Since $\mu \le \widetilde{\mu }$, we get the same estimate with
$\widetilde{\mu }$ replaced by $\mu $. Then from \no{pr.18} we get
\begin{prop}\label{pr2}
Let $\epsilon >0$ be small enough for 
the right hand side of \no{pr.14.5} to be $\le 1$. Then
\ekv{pr.26}
{
P(|F(\alpha )|<\epsilon )\le {\cal O}(1) h^{-N_6}
\frac{\epsilon _0(h)}{h^n}\exp (\frac{h^n}{{\cal O}(1)\epsilon
  _0(h)}\ln \epsilon ).
}
\end{prop}

\begin{remark}\label{pr3}
{\rm In the case when $\widetilde{R}$ has real coefficients, we may assume
that the eigenfunctions $\epsilon _j$ are real, and from the observation after
Proposition \ref{spe2} we see that we can choose $\alpha _0$ above to
be real. The discussion above can then be restricted to the case of
real $\alpha ^1$ and hence to real $\alpha $. We can then introduce
the probability measure $P$ as in \no{pr.16} on the real ball $B_{{\bf
    R}^D}(0,R)$. The subsequent discussion goes through without any
changes, and we still have the conclusion of Proposition \ref{pr2}.}
\end{remark}

\section{End of the proof of the main result}\label{en}
\setcounter{equation}{0}

We now work under the assumptions of Theorem \ref{int1}. For $z$ in a
fixed neighborhood of $\Gamma $,
we rephrase \no{pr.6} as
\ekv{en.1}
{
|\det P_{\delta ,z}| \le \exp \frac{1}{h^n}(\frac{1}{(2\pi )^n}
\iint \ln |p_z| dxd\xi +\epsilon _0(h)),
}
where $\epsilon _0(h)$ is given in \no{pr.8}. 
Moreover, Proposition \ref{pr2} shows that with probability 
\ekv{en.3}
{
\ge 1-{\cal O}(1)h^{-N_6-n}\epsilon _0(h)e^{-\frac{
h^n}{{\cal O}(1)\epsilon _0(h)}\ln \frac{1}{\epsilon }},
}
we have
\ekv{en.4}{
|\det P_{\delta ,z}|\ge \epsilon \exp (\frac{1}{h^n}(\frac{1}{(2\pi
  )^n})
\iint \ln |p_z| dxd\xi ),
}
provided that $\epsilon >0$ is small enough so that
\ekv{en.5}
{
\hbox{The right hand side of \no{pr.14.5} is }\le 1,\,\forall \alpha
^1\in S^{2D-1}.
}

From (\ref{pr.8}) and the subsequent remark we can take
\ekv{en.7}
{
\epsilon _0(h)=C(h^\kappa +h^n\ln \frac{1}{h})(\ln \frac{1}{\tau_0}
+(\ln \frac{1}{h})^2). 
}

\par Write $\epsilon =e^{-\widetilde{\epsilon }/h^n}$,
$\widetilde{\epsilon }=h^n\ln \frac{1}{\epsilon }$. Then \no{en.5}
holds if 
\ekv{en.9}
{
\widetilde{\epsilon }\ge C\epsilon _0(h),
}
for some large constant $C$. \no{en.3}, \no{en.4} can be
rephrased by saying that with probability
\ekv{en.10}
{
\ge 1-{\cal O}(1)h^{-N_6-n}\epsilon _0(h)e^{-\frac{1}{C}\frac{\widetilde{\epsilon }}{\epsilon _0(h)}},
}
we have
\ekv{en.11}
{
|\det P_{\delta ,z}|\ge \exp \frac{1}{h^n}(\frac{1}{(2\pi )^n}\iint
\ln |p_z|dxd\xi -\widetilde{\epsilon }).
}
This is of interest for $\widetilde{\epsilon }$ in the range
\ekv{en.12}
{
\epsilon _0(h)\ll \widetilde{\epsilon }\ll 1.
}

\par Now, let $\Gamma \Subset \Omega $ be connected with smooth
boundary. Recall that $0<\kappa \le 1$ and that 
\ekv{en.13}
{\hbox{(\ref{grny.3}) holds uniformly for all }z\hbox{ in some neighborhood of
  }\partial \Gamma .}
Then the function
\ekv{en.14}
{
\phi (z) =\frac{1}{(2\pi )^n}\iint \ln |p_z|dxd\xi 
}
is continuous and subharmonic in a neighborhood of
$\partial \Gamma $. Moreover it satisfies the assumption (11.37) of
\cite{HaSj} uniformly for $z$ in some neighborhood of $\partial \Gamma $
with $\rho _0$ there equal to $2\kappa $. We shall apply
Proposition 11.5 in \cite{HaSj} to the holomorphic function 
$$
u(z)=\det P_{\delta ,z},
$$
with ``$\epsilon $'' there replaced by $C\widetilde{\epsilon }$ for
$C>0$ sufficiently large and ``$h$'' there replaced by $h^n$. Choose
$0<r\ll 1$ and $z_1,...,z_N\in \partial \Gamma +D(0,\frac{r}{2})$ as
in that proposition, so that 
$$
\partial \Gamma +D(0,r)\subset \cup_1^N D(z_j,2r),\ N\asymp \frac{1}{r}.
$$
Then, according to \no{en.10}, \no{en.11} we know that with
probability
\ekv{en.15}
{
\ge 1-\frac{{\cal O}(1)\epsilon _0(h)}{rh^{N_6+n}}e^{-
\frac{\widetilde{\epsilon }}{{\cal O}(1)\epsilon _0(h)}}
}
we have 
\ekv{en.16}
{
h^n\ln |u(z_j)|\ge \phi (z_j)-\widetilde{\epsilon },\ j=1,...,N.
}

\par In a full neighborhood of $\partial \Gamma $ we also have
\ekv{en.17}
{
h^n\ln |u(z)|\le \phi (z)+C\widetilde{\epsilon }.
}
By Proposition 11.5 in \cite{HaSj} we conclude that with probability
bounded from below as in \no{en.15}
 we have for every $\widehat{M}>0$:
\eekv{en.18}
{
&&|
\#(u^{-1}(0)\cap \Gamma )-\frac{1}{h^n2\pi }\int_\Gamma \Delta \phi L(dz)
|\le
}
{&&
\frac{{\cal O}(1)}{h^n}\left( \frac{\widetilde{\epsilon }}{r}
+{\cal O}_{\widehat{M}}(1)(r^{\widehat{M}}+
\ln (\frac{1}{r}))\mu (\partial \Gamma +D(0,r))
 \right),
}
where $\mu $ denotes the measure $\Delta \phi L(dz)$. Choose $\widehat{M}=1$.

According to Section 10 in \cite{HaSj}, we know that near $\Gamma $,
the measure $\frac{1}{2\pi
}\Delta \phi L(dz)$ is the push forward under $p$ of $(2\pi )^{-n}$
times the symplectic volume element, and we
can replace $\frac{1}{2\pi
}\Delta \phi L(dz)$ by this push forward in \no{en.18}. Moreover $u^{-1}(0)$ is the set
of eigenvalues of $P_\delta $ so we can rephrase \no{en.18} (with $\widehat{M}=1$) as 
\eekv{en.19}
{
&&|
\#(\sigma (P_\delta )\cap \Gamma )-\frac{1}{(2\pi h)^n
}\mathrm{vol\,}(p^{-1}(\Gamma ))
|\le
}
{&&
\frac{{\cal O}(1)}{h^n}\left( \frac{\widetilde{\epsilon }}{r}
+{\cal O}(1)(r+\ln (\frac{1}{r}))\mathrm{vol\,}(p^{-1}(\partial
\Gamma +D(0,r)))
 \right).
}
This concludes the proof of Theorem \ref{int1}, with $P$ replaced by
the slightly more general operator $P_0$.

\section{Appendix: Review of some $h$-pseudo\-differ\-ential 
calculus}\label{app}
\setcounter{equation}{0}

We recall some basic $h$-pseudodifferential calculus on compact
manifolds, including some fractional powers in the spirit of R.~Seeley
\cite{Se}. Recall from \cite{Ho}
that if $X\subset {\bf R}^n$ is open, $0<\rho \le 1$, $m\in {\bf R}$,
then $S^m_\rho (X\times {\bf R}^n)=S^m_{\rho ,1-\rho } (X\times {\bf
  R}^n)$ is defined to be the space of all $a\in C^\infty (X\times
{\bf R}^n)$ such that $\forall K\Subset X,$ $\alpha ,\beta \in {\bf
  N}^n,$ there exists a constant
$C=C(K,\alpha ,\beta )$, such that
\ekv{app.1}
{
|\partial _x^\alpha \partial _\xi ^\beta a(x,\xi )|\le C
\langle \xi \rangle^{m-\rho |\beta |+(1-\rho )|\alpha |},\ (x,\xi )\in
K\times {\bf R}^n.
}
When $a(x,\xi )=a(x,\xi ;h)$ depends on the additional parameter $h\in
]0,h_0]$ for some $h_0>0$, we say that $a\in S^m_\rho (X\times {\bf
  R}^n)$, if \no{app.1} holds uniformly with respect to $h$. For
$h$-dependent symbols, we introduce $S^{m,k}_\rho =h^{-k}S^m_\rho
$. When $\rho =1$ it is customary to suppress the subscript $\rho $.

\par Let now $X$ be a compact $n$-dimensional manifold. We say that 
$R=R_h:{\cal D}'(X)\to C^\infty (X)$ is negligible, and write
$R\equiv 0$, if the distribution-kernel $K_R$ satisfies $\partial
_x^\alpha \partial _y^\beta K_R(x,y)={\cal O}(h^\infty )$ for all
$\alpha ,\beta \in {\bf N}^n$ (when expressed in local coordinates).

We say that an operator $P=P_h:C^\infty (X)\to {\cal D}'(X)$ belongs
to the space $L^{m,k}(X)$ if $\phi P_h\psi $ is negligible for all
$\phi ,\psi \in C^\infty (X)$ with disjoint supports and if for every
choice of local coordinates $x_1,...,x_n$, defined on the open subset
$\widetilde{X}\subset X$ (that we view as a subset of ${\bf R}^n$), we
have on $\widetilde{X}$ for every $u\in C_0^\infty (\widetilde{X})$:
\ekv{app.2}
{
Pu(x)=\frac{1}{(2\pi h)^n}\iint e^{\frac{i}{h}(x-y)\cdot \theta}
a(x,\theta ;h)u(y) dyd\theta +Ku(x),
}
where $a\in S^{m,k}(\widetilde{X}\times {\bf R}^n)$ and $K$ is negligible. 

The correspondence $P \mapsto a$ is not globally well-defined, but the 
various local maps give rise to a bijection
\ekv{app.3}{
L^{m,k}(X)/L^{m-1,k-1}(X) \to S^{m,k}(T^*X)/S^{m-1,k-1}(T^*X),
}
where we notice that $S^{m,k}(T^*X)$ is well-defined in the natural
way. The image $\sigma _P(x,\xi )$ of $P\in L^{m,k}(X)$ is called the
principal symbol. 

Pseudodifferential operators in the above classes map $C^\infty $ to
$C^\infty $ and extend to well-defined operators ${\cal D}'(X)\to
{\cal D}'(X)$. They can therefore be composed with each other: If
$P_j\in L^{m_j,k_j}(X)$, for $j=1,2$, then $P_1\circ P_2\in
L^{m_1+m_2,k_1+k_2}$. Moreover $\sigma _{P_1\circ P_2}(x,\xi )=\sigma
_{P_2}(x,\xi )\sigma _{P_1}(x,\xi )$.

\par We can invert elliptic operators: If $P_h\in L^{m,k}$ is elliptic
in the sense that $|\sigma _{P}(x,\xi )|\ge \frac{1}{C}h^{-k}\langle
\xi \rangle^m$, then $P_h$ is invertible (either as a map on $C^\infty
$ or on ${\cal D}'$) for $h>0$ small enough, and the inverse $Q$
belongs to $L^{-m,-k}$. (If we assume invertibility in the full range
$0<h\le h_0$ then the conclusion holds in that range.) Notice that 
$\sigma _Q(x,\xi )=1/\sigma _P(x,\xi )\in S^{-m,-k}/S^{-m-1,-k-1}$. 

\par The proof of these facts is a routine application of the method of
stationary phase, following for instance the presentation in \cite{GrSj}.

\par Let $\widetilde{R}$ be a positive elliptic 2nd order 
differential operator with smooth coefficients on $X$, self-adjoint
with respect to some smooth positive density on $X$. Let $r(x,\xi )$
be the principal symbol of $\widetilde{R}$ in the classical sense, so
that $r(x,\xi )$ is a homogeneous polynomial in $\xi $ with $r(x,\xi
)\asymp |\xi |^2$. Then $P:=h^2\widetilde{R}$ belongs to $L^{2,0}(X)$
and $\sigma _{h^2\widetilde{R}}=r$. 
\begin{prop}\label{app1} For every $s\in {\bf R}$, we have
$(1+h^2\widetilde{R})^{2s}\in L^{2s,0}$ and the principal symbol is
given by $(1+r(x,\xi ))^s$.
\end{prop}
\begin{proof} It suffices to show this for $s$ sufficiently large
negative. In that case we have \ekv{app.4} {
(1+h^2\widetilde{R})^s=\frac{1}{2\pi i}\int_\gamma
(1+z)^s(z-h^2\widetilde{R})^{-1}dz, } where $\gamma $ is the oriented
boundary of the sector $\mathrm{arg\,}(z+\frac{1}{2})<\pi /4$. For
$z\in \gamma $, we write
$$
(z-h^2\widetilde{R})=|z|(\frac{z}{|z|}-\widetilde{h}^2\widetilde{R}),\
\widetilde{h}=\frac{h}{|z|^{1/2}},
$$
and notice that $\frac{z}{|z|}-\widetilde{h}^2\widetilde{R}\in
L^{2,0}$ is elliptic when we regard $\widetilde{h}$ as the new
semi-classical parameter. By self-adjointness and positivity we know
that this operator is invertible, so
$(\frac{z}{|z|}-\widetilde{h}^2\widetilde{R})^{-1}\in L^{-2,0}$, and
for every system of local coordinates the symbol (in the sense of
$\widetilde{h}$-pseudodifferential operators) is \ekv{app.4.5}
{\frac{1}{\frac{z}{|z|}-r(x,\xi )}+a,\quad a\in S^{-3,-1}.}  The
symbol of $(z-h^2\widetilde{R})^{-1}$ as an $h$-pseudodifferential operator
is therefore \ekv{app.5} { \frac{1}{|z|(\frac{z}{|z|}-r(x,\frac{\xi
}{|z|^{1/2}}))}+\frac{1}{|z|}a(x,\frac{\xi }{|z|^{1/2}}).  } Here the
first term simplifies to $(z-r(x,\xi ))^{-1}$ and the corresponding
contribution to \no{app.4} has the symbol $(1+r(x,\xi ))^s$.

\par The contribution from the remainder in \no{app.5} to the symbol
in \no{app.4} is
$$b(x,\xi ):=\frac{1}{2\pi i}\int_\gamma
\frac{(1+z)^s}{|z|}a(x,\frac{\xi }{|z|^{1/2}})dz,$$ where we will use
the estimate
\ekv{app.5.5}{
\partial _x^\alpha \partial _\xi^\beta \frac{1}{|z|}
a(x,\frac{\xi}{|z|^{1/2}})={\cal O}(\frac{h}{|z|^{(3+|\beta |)/2}}
\langle \frac{\xi
}{|z|^{1/2}}\rangle^{-3-|\beta |})={\cal O}(h)(|z|+|\xi
|^2)^{-\frac{1}{2}(3+|\beta |)}.
}
Thus, \ekv{app.6} {
\partial _x^\alpha \partial _\xi ^\beta b={\cal O}(1)h\int _\gamma
|z|^s(|z|+|\xi
|^2)^{-\frac{1}{2}(3+|\beta |)}
|dz|.  }

\par In a region $|\xi |={\cal O}(1)$, we get
$$
\partial _x^\alpha \partial _\xi ^\beta b ={\cal O}(1).
$$

\par In the region $|\xi |\gg 1$ 
shift the contour $\gamma $ in (\ref{app.4}) to the oriented boundary
of the sector $\mathrm{arg\,}(z+\frac{1}{2}|\xi |^2)<\frac{\pi
}{4}$. Then we get
\no{app.6} for the shifted contour and the integral can now be
estimated by
$$
{\cal O}(h)\int_{|\xi |^2/C}^\infty t^{s-\frac{3}{2}-\frac{|\beta
    |}{2}} dt = {\cal O}(h|\xi |^{2s-1-|\beta |}).
$$
The proposition follows.
\end{proof}

\end{document}